\documentclass{amsart}
\usepackage{amsmath,amssymb,latexsym,amscd}
\usepackage{tikz}
\usetikzlibrary{arrows,matrix}
\usepackage[margin=1in]{geometry}
\input xy
\xyoption{2cell}
\xyoption{all}
 
\numberwithin{equation}{section}

\newtheorem{theorem}{Theorem}

\newtheorem{conjecture}[theorem]{Conjecture}
\newtheorem{corollary}[theorem]{Corollary}
\newtheorem{definition}[theorem]{Definition}

\newtheorem{lemma}[theorem]{Lemma}
\newtheorem{proposition}[theorem]{Proposition}
\newtheorem{remark}[theorem]{Remark}

\numberwithin{theorem}{section}

\newcommand{\bfa}{\mathbf{a}}

\newcommand{\bfb}{\mathbf{b}}

\newcommand{\bfc}{\mathbf{c}}
\newcommand{\bfd}{\mathbf{d}}

\newcommand{\bfe}{\mathbf{e}}
\newcommand{\bff}{\mathbf{f}}
\newcommand{\bfg}{\mathbf{g}}
\newcommand{\bfi}{\mathbf{i}}
\newcommand{\bfj}{\mathbf{j}}
\newcommand{\bfk}{\mathbf{k}}

\newcommand{\bfl}{\mathbf{l}}

\newcommand{\bfp}{\mathbf{p}}

\newcommand{\bft}{\mathbf{t}}
\newcommand{\bfv}{\mathbf{v}}
\newcommand{\bfw}{\mathbf{w}}
\newcommand{\bfX}{\mathbf{X}}
\newcommand{\bfx}{\mathbf{x}}
\newcommand{\bfy}{\mathbf{y}}

\newcommand{\cA}{\mathcal{A}}

\newcommand{\cC}{\mathcal{C}}

\newcommand{\cF}{\mathcal{F}}

\newcommand{\cK}{\mathcal{K}}

\newcommand{\cP}{\mathcal{P}}

\newcommand{\cT}{\mathcal{T}}

\newcommand{\FF}{\mathbb{F}}

\newcommand{\TT}{\mathbb{T}}
\newcommand{\ZZ}{\mathbb{Z}}

\newcommand{\Aut}{\operatorname{Aut}}
\newcommand{\coker}{\operatorname{coker}}
\renewcommand{\dim}{\operatorname{dim}}

\newcommand{\End}{\operatorname{End}}
\newcommand{\Ext}{\operatorname{Ext}}
\renewcommand{\ker}{\operatorname{ker}}
\newcommand{\Hom}{\operatorname{Hom}}
\newcommand{\Id}{\operatorname{Id}}
\newcommand{\im}{\operatorname{im}}
\renewcommand{\max}{\operatorname{max}}
\newcommand{\Mod}{\operatorname{Mod}}

\newcommand{\rad}{\operatorname{rad}}
\newcommand{\Rep}{\operatorname{Rep}}
\newcommand{\rep}{\operatorname{rep}}
\newcommand{\soc}{\operatorname{soc}}
\newcommand{\spa}{\operatorname{span}}
\newcommand{\supp}{\operatorname{supp}}
\newcommand{\Tr}{\operatorname{Tr}}

\newcommand{\half}{\frac{1}{2}}

\renewcommand{\eqref}[1]{{\rm (\ref{#1})}}

\title{Quantum Cluster Characters}
\author{Dylan Rupel}
\address{\noindent Department of Mathematics, University of Oregon, Eugene, OR 97403}
\email{drupel@uoregon.edu}

\begin{document}

\begin{abstract}
 Let $\FF$ be a finite field and $(Q,\bfd)$ an acyclic valued quiver with associated exchange matrix $\tilde{B}$.  We follow Hubery's approach \cite{hub1} to prove our main conjecture of \cite{rupel}: the quantum cluster character gives a bijection from the isoclasses of indecomposable rigid valued representations of $Q$ to the set of non-initial quantum cluster variables for the quantum cluster algebra $\cA_{|\FF|}(\tilde{B},\Lambda)$.  As a corollary we find that, for any rigid valued representation $V$ of $Q$, all Grassmannians of subrepresentations $Gr_\bfe^V$ have counting polynomials.
\end{abstract}

\keywords{quantum cluster algebra, valued quiver Grassmannian}

 \maketitle

 \makeatletter
 \renewcommand{\@evenhead}{\tiny \thepage \hfill  D.~RUPEL \hfill}

 \renewcommand{\@oddhead}{\tiny \hfill QUANTUM CLUSTER CHARACTERS \hfill \thepage}
 \makeatother

 \tableofcontents

\section{Introduction and Main Results}\label{sec:intro}

 Cluster algebras were introduced by Fomin and Zelevinsky \cite{fz1} in anticipation that cluster monomials would be contained in Lusztig's dual canonical basis for quantized coordinate rings of algebraic groups.  Thus a serious attempt was made to describe the cluster monomials.  In \cite{fz2}, they establish a simple bijection between the cluster variables and almost positive roots in an associated root system, thus proving the well known $A-G$ finite-type classification for cluster algebras.  

 It has been recognized \cite{gab},\cite{kac},\cite{ringel1} that the indecomposable representations of an associated (valued) quiver $Q$ were also in bijection with a certain root datum, namely the (strictly) positive roots.  To properly explain these bijections and fully understand the role of the negative simple roots, the authors of \cite{bmrrt} introduce the cluster category $\cC_Q$ in which the indecomposable representations are exactly in bijection with the almost positive roots.  Moreover, they observe that there is a bijection between the cluster-tilting objects of $\cC_Q$ and the clusters of the cluster algebra $\cA(Q)$.

 Extending this, Caldero and Chapoton \cite{caldchap} introduce cluster characters describing the initial cluster expansion of all cluster variables/monomials explicitly as generating functions of Euler characteristics of Grassmannians of submodules in the corresponding quiver representations.  Following the BMRRT approach, they show in finite types that the mutation operation for cluster-tilting objects coincides with the seed-mutation of Fomin and Zelevinsky, thus establishing a relationship between rigid objects of the cluster category and cluster monomials.  Then in \cite{caldkell} Caldero and Keller generalize these results to give cluster characters categorifying all cluster algebras with acyclic, skew-symmetric exchange matrix.  

 In both \cite{caldchap} and \cite{caldkell} the authors recognize a similarity between multiplication in the cluster algebra and the multiplication in the dual Hall algebra.  Building on these observations, Hubery in his preprint \cite{hub1} works toward extending the above results to the acyclic skew-symmetrizable case, replacing the category of representations of an acyclic quiver $Q$ with the hereditary category of finite dimensional modules over a species.  However, his approach follows a specialization argument requiring the existence of Hall polynomials, which is still an open conjecture.  Nevertheless the computations of \cite{hub1} are valid and play a vital role in the proof of our main theorem.  As a corollary to our main result we obtain the existence of counting polynomials for rigid quiver Grassmannians.  Hubery's specialization argument only requires that certain sums of Hall numbers, namely the number of points in the quiver Grassmannian, are given by polynomials.  Thus Corollary~\ref{poly_count} resolves Hubery's specialization argument.

 In \cite{rupel}, we define a quantum cluster character for the category $\rep_\FF(Q,\bfd)$ of representations of the valued quiver $(Q,\bfd)$ over the finite field $\FF$.  This character assigns to a valued representation $V$ of $Q$ the element $X_V$ in the quantum torus $\cT_{\Lambda,|\FF|}$, given by
 \begin{equation}\label{eq:qcc_formula}
  X_V=\sum\limits_{E\subset V} |\FF|^{-\half\langle E,V/E\rangle} X^{-[E]^*-{}^*[V/E]}
 \end{equation}
 where $\Lambda$ is compatible with $(Q,\bfd)$, $\cT_{\Lambda,|\FF|}$ is the quantum torus defined in section~\ref{sec:qca}, $\langle\cdot,\cdot\rangle$ denotes the Ringel-Euler form of the category $\rep_\FF(Q,\bfd)$, and ${}^*$ denotes certain left and right duals with respect to the Ringel-Euler form, see section~\ref{sec:valued_quiver}.  For a valued representation $V$ of $Q$ we will denote its class in the Grothendieck group $\cK(Q,\bfd)$ by the corresponding bold face lower case letter $\bfv$.  Clearly, formula~\eqref{eq:qcc_formula} is equivalent to 
 \[X_V=\sum\limits_{\bfe\in\cK(Q,\bfd)} |\FF|^{-\half\langle \bfe,\bfv-\bfe\rangle}|Gr_\bfe^V| X^{-\bfe^*-{}^*(\bfv-\bfe)}\]
 where $Gr_\bfe^V$ is the Grassmannian of subrepresentations in $V$ of type $\bfe$.  

 Here we present the first main result of the paper which was conjectured in \cite{rupel}.
 \begin{theorem}\cite[Conjecture 1.10]{rupel}\label{th:main}
  The quantum cluster character $V\mapsto X_V$ defines a bijection from exceptional valued $\FF$-representations $V$ of $Q$ to non-initial quantum cluster variables of the quantum cluster algebra $\cA_{|\FF|}(\tilde{B},\Lambda)$.
 \end{theorem} 

 The quantum cluster character was previously known \cite{rupel} to give cluster variables for all finite type valued quivers, for rank 2 valued quivers, and for those representations of acyclic valued quivers which can be obtained from simple representations by a sequence of reflection functors.  Qin \cite{qin} proved this conjecture for acyclic equally valued quivers (i.e. $d_i=1$ for all $i$).

 Note that there is a unique isomorphism class for each exceptional valued representation and thus the isomorphism classes of rigid objects in the Grothendieck group $\cK(Q,\bfd)$ are independent of the choice of ground field $\FF$.  Theorem~\ref{th:main} together with the quantum Laurent phenomenon from \cite{berzel} imply the following result.  
 \begin{corollary}\label{poly_count}
  Let $V$ be a rigid valued representation of $Q$.  Then for any $\bfe\in\cK(Q,\bfd)$ the Grassmannian $Gr_\bfe^V$ has a counting polynomial, which we denote by $P_\bfe^\bfv(q)$.
 \end{corollary}

 For equally valued quivers Qin \cite{qin} shows that these counting polynomials have nonnegative integer coefficients.  In a recent preprint \cite{ef}, Efimov shows that in addition the counting polynomials are unimodular, i.e. they are a shifted sum of bar-invariant $q$-numbers.  We conjecture that these properties always hold.
 \begin{conjecture}
  For any $\bfe$ and rigid $V$ the polynomial $P_\bfe^\bfv(q)\in\ZZ[q]$ is unimodular and has nonnegative integer coefficients.
 \end{conjecture}

 Formula~\eqref{eq:qcc_formula} makes sense for any category $\cC$ that has finitely many simple objects and any object $V$ of $\cC$ which has finitely many subobjects.  Therefore we hope that by choosing an appropriate category, e.g. a cluster category or representations of a valued quiver with potential, this formula will give all quantum cluster variables and monomials.  In this case, the same machinery would imply polynomiality of the cardinality of Grassmannians in the category $\cC$. 

 While this manuscript was in its final stages of preparation Ding and Sheng \cite{dingsheng} proved multiplication theorems for quantum cluster characters extending those of \cite{dingxu}.  In this work Ding and Sheng deduce results about bases of quantum cluster algebras of finite type or rank 2.  We are curious what can be said about bases in more general acyclic quantum cluster algebras.

 The paper is organized as follows.  Section~\ref{sec:valued_quiver} contains definitions and results related to valued quivers and section~\ref{sec:qca} recalls the definition of quantum cluster algebras.  Sections \ref{sec:tilting}, \ref{sec:tilt_matrix}, and \ref{sec:exch_mut} recall Hubery's results on mutations of local tilting representations and associated exchange matrices.  In section~\ref{sec:mult_theorems}, we prove multiplication formulas similar to those from \cite{caldkell}, \cite{hub1}, \cite{qin}, and \cite{dingxu}.  The associated commutation matrix is computed in section~\ref{sec:comm}.

 \subsection{Notation}  We always use bold face letters to denote an $n$-tuple of integers.  Also if a valued representation has been named, we will use the corresponding boldface lowercase letter to denote it's dimension vector, e.g. the dimension vector of $V$ is $\bfv=(v_1,\ldots,v_n)$.

\section{Acknowledgements}

 The author would like to thank his PhD advisor Arkady Berenstein for encouraging him to write this paper.  A portion of this article was written while the author was a guest at the Banff International Research Station, he would like to offer his sincerest gratitude to BIRS for the invitation to work in such a beautiful setting.  The author would also like to thank Hugh Thomas and Ralf Schiffler for stimulating discussions related to cluster categories.

\section{Valued Quiver Representations}\label{sec:valued_quiver}

 Let $\FF$ be a finite field and write $\overline{\FF}$ for an algebraic closure of $\FF$.  For each positive integer $k$ denote by $\FF_k$ the degree $k$ extension of $\FF$ in $\overline{\FF}$.  Note that the largest subfield of $\overline{\FF}$ contained in both $\FF_k$ and $\FF_\ell$ is $\FF_{gcd(k,\ell)}=\FF_k\cap\FF_\ell$.  If $k|\ell$ we will fix a basis of $\FF_\ell$ over $\FF_k$ and thus we may freely identify $\FF_\ell$ as a vector space over $\FF_k$.

 Fix an integer $n$ and let $Q=\{Q_0, Q_1, s, t\}$ be a quiver with vertices $Q_0=\{1,2,\ldots, n\}$ and arrows $Q_1$, where we denote by $s(a)$ and $t(a)$ the source and target of an arrow 
 \begin{center}
  \begin{tikzpicture}
   \matrix (m) [matrix of math nodes, row sep=1em, column sep=2.5em, text height=1.5ex, text depth=0.25ex]
    {s(a) & t(a)\,.\\};
   \path[->]
    (m-1-1) edge node[auto] {$a$} (m-1-2);
  \end{tikzpicture}
 \end{center}
 Let $\bfd: Q_0\to \ZZ_{>0}$ be a collection of valuations associated to the vertices of $Q$, where we denote the image of a vertex $i$ by $d_i$.  We will call the pair $(Q,\bfd)$ a ``valued quiver''.  Since the valuations will be fixed for all time we will sometimes drop the $\bfd$ from our notation.

 Define a representation $V=(\left\{V_{i}\right\}_{i\in Q_0},\left\{\varphi_{a}\right\}_{a\in Q_1})$ of $(Q,\bfd)$, or a ``($\bfd$-)valued representation'' of $Q$, by assigning an $\FF_{d_i}$-vector space $V_i$ to each vertex $i\in Q_0$ and an $\FF_{gcd(d_{s(a)}, d_{t(a)})}$-linear map
 \begin{center}
  \begin{tikzpicture}
   \matrix (m) [matrix of math nodes, row sep=1em, column sep=2.5em, text height=1.5ex, text depth=0.25ex]
    {\varphi_a\ :\ V_{s(a)} & V_{t(a)}\\};
   \path[->]
    (m-1-1) edge (m-1-2);
  \end{tikzpicture}
 \end{center} 
 to each arrow $a\in Q_1$.  Let $W=(\left\{W_{i}\right\}_{i\in Q_0}, \left\{\psi_{a}\right\}_{a\in Q_1})$ denote another valued representation of $Q$ and define a morphism
 \begin{center}
  \begin{tikzpicture}
   \matrix (m) [matrix of math nodes, row sep=1em, column sep=2.5em, text height=1.5ex, text depth=0.25ex]
    {\theta\ :\ V & W\\};
   \path[->]
    (m-1-1) edge (m-1-2);
  \end{tikzpicture}
 \end{center} 
 to be a collection $\theta=\{\theta_i\in \Hom_{\FF_{d_i}}(V_i,W_i)\}_{i\in Q_0}$ such that $\theta_{t(a)}\circ\varphi_a=\psi_a\circ\theta_{s(a)}$, i.e. the following diagram commutes, for all $a\in Q_1$:
 \begin{center}
  \begin{tikzpicture}
   \matrix (m) [matrix of math nodes, row sep=3em,column sep=2.5em, text height=1.5ex, text depth=0.25ex]
    {V_{s(a)} & V_{t(a)}\\ W_{s(a)} & W_{t(a)}\,.\\};
   \path[->,font=\scriptsize]
    (m-1-1) edge node[auto] {$\varphi_a$} (m-1-2) edge node[auto] {$\theta_{s(a)}$} (m-2-1)
    (m-1-2) edge node[auto] {$\theta_{t(a)}$} (m-2-2)
    (m-2-1) edge node[auto] {$\psi_a$} (m-2-2);  
  \end{tikzpicture}
 \end{center}
 Thus we have a category $\Rep_\FF(Q,\bfd)$ of all finite dimensional valued representations of $Q$.  For valued representations $V$ and $W$ we define their direct sum via 
 \[V\oplus W=(\left\{V_i\oplus W_i\right\}_{i\in Q_0},\left\{\varphi_a\oplus\psi_a\right\}_{a\in Q_1}).\]
 Then one easily checks that $\Rep_\FF(Q,\bfd)$ is an Abelian category where kernels and cokernels are taken vertex-wise.  

 For each $i\in Q_0$ denote by $S_i$ the simple valued representation associated to vertex $i$, i.e. we assign the vector space $\FF_{d_i}$ to vertex $i$ and the zero vector space to every other vertex.  We will write $\cK(Q)$ for the Grothendieck group of $\Rep_\FF(Q,\bfd)$, that is $\cK(Q)$ is the free Abelian group generated by the isomorphism classes $[V]$ for $V\in\Rep_\FF(Q,\bfd)$ subject to the relations $[V]=[U]+[W]$ whenever there is a short exact sequence
 \begin{center}
  \begin{tikzpicture}
   \matrix (m) [matrix of math nodes, row sep=1em, column sep=2.5em, text height=1.5ex, text depth=0.25ex]
    {0 & U & V & W & 0\,.\\};
   \path[->]
    (m-1-1) edge (m-1-2)
    (m-1-2) edge (m-1-3)
    (m-1-3) edge (m-1-4)
    (m-1-4) edge (m-1-5);
  \end{tikzpicture}
 \end{center}
 From now on we assume that the quiver $Q$ is acyclic, that is $Q$ contains no non-trivial paths which begin and end at the same vertex.  Let $\alpha_i\in\cK(Q)$ denote the class of simple valued representation $S_i$.  Every representation $V\in\Rep_\FF(Q,\bfd)$ has a finite filtration with simple quotients, so we may write $[V]$ as a linear combination of the $\alpha_i$.  Thus we may identify $[V]$ with the ``dimension vector'' of $V$: $\bfv=(\dim_{\FF_{d_i}} V_i)_{i=1}^n\in\ZZ^n$, where we adopt the convention that for any named representation we will use the same bold face letter to denote its dimension vector.  In particular, taking the $\alpha_i$ as a basis we may identify $\cK(Q)$ with the free Abelian group $\ZZ^n$.

 Now we introduce some terminology that will be useful for describing valued representations.  A valued representation $V$ is called ``indecomposable'' if $V=U\oplus W$ implies $U=0$ or $W=0$.  The category $\Rep_\FF(Q,\bfd)$ is Krull-Schmidt, that is every representation can be written uniquely as a direct sum of indecomposable representations.  A representation $V$ will be called ``basic'' if each indecomposable summand of $V$ appears with multiplicity one.  We will call $V$ ``rigid'' if $\Ext(V,V)=0$.  If $V$ is both rigid and indecomposable then we will say $V$ is ``exceptional''.  Our main concern will be with basic rigid representations.  In particular, we are most interested in basic rigid representations $V$ satisfying the following locality condition: the number of vertices in the ``support'' of $V$, $\supp(V)=\{i\in Q_0 | V_i\ne0\}$, is equal to the number of indecomposable summands of $V$.  These representations will be called ``local tilting representations''.  A representation V is called ``sincere'' if $\supp(V)=Q_0$.  A sincere local tilting representation will simply be called a ``tilting representation''.  We justify this terminology in Section~\ref{sec:tilting}.  Note that we may identify an insincere representation $V$ with a sincere representation of the full subquiver $Q^V$ of $Q$ with vertices $Q^V_0=\supp(V)$.  Then a local tilting representation $V$ may be considered as a tilting representation for $Q^V$, this explains the adjective local. 

 Just as with ordinary quivers it is useful to consider an equivalent category of modules over the path algebra.  The analog of the path algebra for valued quivers is the notion of an $\FF$-species.  To describe this correspondence we will need to introduce some notation.  For $i,j\in Q_0$ denote by $n_{ij}$ the number of arrows connecting vertices $i$ and $j$ and note that these arrows are either all of the form $i\to j$ or all of the form $j\to i$.  Define a matrix $B_Q=B_{(Q,\bfd)}=(b_{ij})$ by 
 \[b_{ij}=\begin{cases} n_{ij}d_j/gcd(d_i,d_j) & \text{ if $i\to j\in Q_1$;}\\-n_{ij}d_j/gcd(d_i,d_j) & \text{ if $j\to i\in Q_1$;}\\ 0 & \text{ otherwise.}\end{cases}\]
 Let $d_{ij}=gcd(d_i,d_j)$ and $d^{ij}=lcm(d_i,d_j)$.  Notice that since $d_id_j/d_{ij}=d^{ij}$ we have $d_ib_{ij}=n_{ij}d^{ij}=-d_jb_{ji}$ whenever $i\to j\in Q_1$.  Thus we see that $DB_Q$ is skew-symmetric where $D=diag(d_i)$, in other words $B_Q$ is ``skew-symmetrizable''.  
 Note that we may recover $Q$ from the matrix $B_Q$ as the quiver with $gcd(|b_{ij}|,|b_{ji}|)$ arrows from $i$ to $j$ whenever $b_{ij}>0$.

 Now we are ready to define an $\FF$-species $\Gamma_Q$ whose modules are the same as valued representations of $Q$.  Define $\Gamma_0=\prod_{i=1}^n \FF_{d_i}$ and $\Gamma_1=\bigoplus_{b_{ij}>0} \Gamma_{ij}$ where $\Gamma_{ij}:=\FF_{d_ib_{ij}}$.  We define an $\FF_{d_i}-\FF_{d_j}$-bimodule structure on $\Gamma_{ij}$ by setting $\Gamma_{ij}=\FF^{d_{ij}}\otimes_{\FF} \FF_{\delta^{ij}}\cong\bigoplus_{k=1}^{d_{ij}} \FF_{\delta^{ij}}$, where $\frac{d_ib_{ij}}{\delta^{ij}}=gcd(|b_{ij}|,|b_{ji}|)=d_{ij}$.  Notice that $\FF_{\delta^{ij}}$ contains both $\FF_{d_i}$ and $\FF_{d_j}$ and thus we have such a bimodule structure.  Thus we may consider $\Gamma_1$ as a bimodule over the semisimple algebra $\Gamma_0$ and form the tensor algebra $\Gamma_Q=T_{\Gamma_0}(\Gamma_1)$ of $\Gamma_1$ over $\Gamma_0$.  

 A module $X$ over $\Gamma_Q$ is given by an $\FF_{d_i}$-vector space $X_i$ for each vertex $i$ and an $\FF_{d_j}$-linear map $\theta^X_{ij}: X_i\otimes_{\FF_{d_i}} \Gamma_{ij} \to X_j$ whenever $b_{ij}>0$, see \cite{hub2} for more details.  A morphism of $\Gamma_Q$-modules $f:X\to Y$ is a collection $\{f_i\}_{i\in Q_0}$, with $f_i:X_i\to Y_i$ an $\FF_{d_i}$-linear map, such that $\theta_{ij}^Y(f_i\otimes id)=f_j\theta_{ij}^X$.  

 \begin{proposition}\cite{rupel}\label{prop:equiv}
  The categories $\Rep_{\FF} (Q,\bfd)$ and $\Mod~\Gamma_Q$ are equivalent.
 \end{proposition}
 \noindent Thus we may apply all results concerning $\FF$-species to valued representations.

 It follows from Proposition~\ref{prop:equiv} that $\Rep_\FF(Q,\bfd)$ is ``hereditary'', that is for all $V,W\in\Rep_\FF(Q,\bfd)$ we have $\Ext^i(V,W)=0$ for every $i\ge2$.  For valued representations $V,W\in\Rep_\FF(Q,\bfd)$ we define $\langle V,W\rangle = \dim_\FF \Hom(V,W)-\dim_\FF\Ext^1(V,W)$.  Using the long exact sequence on $\Ext$ coming from a short exact sequence of representations, it is easy to see that $\langle V,W\rangle$ only depends on the classes of $V$ and $W$ in $\cK(Q)$.  Thus we obtain a bilinear form $\langle\cdot,\cdot\rangle:\cK(Q)\times\cK(Q)\to\ZZ$ known as the ``Ringel-Euler form''.  Note that by the bilinearity of the Ringel-Euler form, it suffices to compute it on the basis of $\cK(Q)$:
 \begin{equation}\label{eq:euler}
  \langle\alpha_i,\alpha_j\rangle=\begin{cases} d_i & \text{ if $i=j$,}\\ -[d_ib_{ij}]_+ & \text{ if $i\ne j$}, \end{cases}
 \end{equation}
 where we write $[b]_+=\max(0,b)$.  

 We will write $\alpha_i^\vee=\frac{1}{d_i}\alpha_i$ and remark that the skew-symmetrizability of $B_Q$ implies $\langle\alpha_i^\vee,\alpha_j\rangle$ and $\langle\alpha_i,\alpha_j^\vee\rangle$ are integers for all $i$ and $j$. Then for $\bfe \in \cK(Q)$ define ${}^*\bfe,\bfe^*\in \cK(Q)$ by 
 \[{}^*\bfe=\sum_{i=1}^n \langle \alpha_i^\vee,\bfe\rangle\alpha_i,~\bfe^*=\sum_{i=1}^n \langle\bfe,\alpha_i^\vee\rangle \alpha_i.\]
 We define two matrices $B_Q^-=(b_{ij}^-)$ and $B_Q^+=(b_{ij}^+)$ as follows:
 \begin{align*}
  b_{ij}^-&=\langle\alpha_i^\vee,\alpha_j\rangle=\begin{cases} 1 & \text{ if $i=j$;}\\ -[b_{ij}]_+ & \text{ if $i\ne j$;} \end{cases}\\
  b_{ij}^+&=\langle\alpha_j,\alpha_i^\vee\rangle=\begin{cases} 1 & \text{ if $i=j$;}\\ -[-b_{ij}]_+ & \text{ if $i\ne j$.} \end{cases}
 \end{align*}
 It is clear from the definitions that viewing $\bfe$ as an element of $\ZZ^n$ we have
 \begin{equation}\label{B_matrices}
  {}^*\bfe=B_Q^-\bfe \text{ and } \bfe^*=B_Q^+\bfe.
 \end{equation}
 Also note from equation~\eqref{eq:euler} and the skew-symmetrizability of $B_Q$ that we have
 \[B_Q^+=D(B_Q^-)^tD^{-1} \text{ and } B_Q^+-B_Q^-=B_Q.\] 

 Suppose there exists an $n\times n$ skew-symmetrizable matrix $\Lambda=(\lambda_{ij})$ such that $B_Q^t\Lambda=D$ and write $\Lambda(\cdot,\cdot):\ZZ^n\times\ZZ^n\to \ZZ$ for the associated skew-symmetric bilinear form.  As in Zelevinsky's Oberwolfach talk on quantum cluster algebras \cite{zel}, we may always replace the quiver $Q$ by $\tilde{Q}$, where we attach principal frozen vertices to the valued quiver $Q$, to guarantee that such a matrix $\Lambda$ exists.  Note that the compatibility condition for $B_Q$ and $\Lambda$ implies $\Lambda(\bfb^i,\alpha_j)=d_i\delta_{ij}$ where $\bfb^i$ denotes the $i^{th}$ column of $B_Q$.  

 For a valued representation $V$ define the ``socle'' $\soc V$ to be the sum of all simple subrepresentations of $V$ and the ``radical'' $\rad V$ to be the intersection of all maximal subrepresentations of $V$.  We record the following identities for use in Section~\ref{sec:qcc}. 
 \begin{lemma}\label{lem:identities}\mbox{}
  \begin{enumerate}
   \item \label{id2} For any $\bfc\in\cK(Q)$, $\Lambda(\bfb^i,{}^*\bfc)=\langle \alpha_i,\bfc\rangle$ and $\Lambda(\bfc^*,\bfb^j)=-\langle \bfc,\alpha_j\rangle$.
   \item \label{id3} For any $\bfb,\bfc\in\cK(Q)$, $\Lambda(\bfb^*-{}^*\bfb,\bfc^*-{}^*\bfc)=\langle\bfc,\bfb\rangle-\langle\bfb,\bfc\rangle$.
   \item \label{id4} For any $\bfb,\bfc,\bfv,\bfw\in\cK(Q)$. $\Lambda(-\bfb^*-{}^*(\bfv-\bfb),-\bfc^*-{}^*(\bfw-\bfc))=\Lambda({}^*\bfv,{}^*\bfw)-\langle\bfc,\bfv-\bfb\rangle+\langle\bfb,\bfw-\bfc\rangle$.
   \item \label{id5} For any injective valued representation $I$, $[\soc I]={}^*\bfi$, where $\bfi=[I]\in\cK(Q)$.
   \item \label{id6} For any projective valued representation $P$, $[P/\rad P]=\bfp^*$, where $\bfp=[P]\in\cK(Q)$.
  \end{enumerate}
 \end{lemma}
 \begin{proof}
 The identities in \eqref{id2} are a direct consequence of the compatibility of $B_Q$ and $\Lambda$.  The identity in \eqref{id3} follows immediately from \eqref{B_matrices} and \eqref{id2}.  The identity in \eqref{id4} can easily be obtained from \eqref{id2} and \eqref{id3}.  There is a unique injective hull and projective cover for each simple representation, the identities in \eqref{id5} and \eqref{id6} follow.
 \end{proof} 

 Now we recall useful functors acting on $\Mod\Gamma_Q$.  By the equivalence of categories in Proposition~\ref{prop:equiv} these functors can be transported to the category $\Rep_\FF(Q,\bfd)$.

 Let $D:\Mod\Gamma_Q\to\Mod\Gamma_{Q^{op}}$ denote the standard $\FF$-linear duality, that is $D=\Hom_\FF(-,\FF)$, where $Q^{op}$ denotes the quiver obtained from $Q$ by reversing all arrows.  We define the ``Nakayama functor'' $\nu:\Mod\Gamma_Q\to\Mod\Gamma_Q$ to be the composition $D\Hom_{\Gamma_Q}(-,\Gamma_Q)$.  It is well-known that this functor defines an equivalence of categories from the full subcategory of $\Mod\Gamma_Q$ consisting of projective objects to the full subcategory consisting of injective objects.

 Our goal is to define the Auslander-Reiten translation.  Let $V$ be a module over $\Gamma_Q$ and consider the minimal projective resolution
 \begin{center}
  \begin{tikzpicture}
   \matrix (m) [matrix of math nodes, row sep=1em, column sep=2.5em, text height=1.5ex, text depth=0.25ex]
    {0 & P_1 & P_0 & V & 0\,.\\};
   \path[->]
    (m-1-1) edge (m-1-2)
    (m-1-2) edge (m-1-3)
    (m-1-3) edge (m-1-4)
    (m-1-4) edge (m-1-5);
  \end{tikzpicture}
 \end{center}
 The ``Auslander-Reiten translation'' $\tau=D\Tr:\Mod\Gamma_Q\to \Mod\Gamma_Q$ is defined on a representation $V$ via the following induced diagram:
 \begin{center}
  \begin{tikzpicture}
   \matrix (m) [matrix of math nodes, row sep=1em, column sep=2.5em, text height=1.5ex, text depth=0.25ex]
    {0 & \Hom(P_0,\Gamma_Q) & \Hom(P_1,\Gamma_Q) & \Tr(V) & 0\,.\\};
   \path[->]
    (m-1-1) edge (m-1-2)
    (m-1-2) edge (m-1-3)
    (m-1-3) edge (m-1-4)
    (m-1-4) edge (m-1-5);
  \end{tikzpicture}
 \end{center}
 The Auslander-Reiten translation defines an equivalence between the full subcategories $\Mod_P\Gamma_Q$ and $\Mod_I\Gamma_Q$ consisting of objects with no projective summands and no injective summands respectively.  The following ``Auslander-Reiten formulas'' will be essential in the computations to follow.
 \begin{proposition}\label{prop:AR}
  For any representations $V,W$, there exist isomorphisms:
  \begin{align}
   \label{eq:AR} D\Hom(V,\tau W)\cong\Ext^1(W,V),\\
   \nonumber D\Ext^1(V,\tau W)\cong\Hom(W,V).
  \end{align}
  In particular, we have $\langle \bfv,\tau\bfw\rangle=-\langle\bfw,\bfv\rangle$.
 \end{proposition}

 \subsection{Tilting Theory for $\Mod\Gamma_Q$}\label{sec:tilting} The main result of this article uses the classical theory of tilting modules over an $\FF$-species.  We will present those results necessary to define a mutation operation for local tilting representations.  Our main result will show that this mutation operation corresponds with the Berenstein-Zelevinsky mutations of quantum seeds that we will present in Section~\ref{sec:qca}.  We freely abuse the equivalence from Proposition~\ref{prop:equiv} to go between representations of $(Q,\bfd)$ and modules over $\Gamma_Q$.

 We begin with the definition of a tilting module for $\Gamma_Q$.  For a representation $T$ we write $add(T)$ for the full additive subcategory of $\Rep_\FF(Q,\bfd)$ generated by the indecomposable summands of $T$.  Then a $\Gamma_Q$-module $T$ is called ``tilting'' if there is a coresolution of $\Gamma_Q$ of the form
 \begin{center}
  \begin{tikzpicture}
   \matrix (m) [matrix of math nodes, row sep=3em, column sep=2.5em, text height=1.5ex, text depth=0.25ex]
    {0 & \Gamma_Q & T_0 & T_1 & 0\\};
   \path[->]
    (m-1-1) edge (m-1-2)
    (m-1-2) edge (m-1-3)
    (m-1-3) edge (m-1-4)
    (m-1-4) edge (m-1-5);  
  \end{tikzpicture}
 \end{center}
 with $T_0,T_1\in add(T)$.  In order to define mutation of local tilting representations we need to recall two well known results from the representation theory of $\Gamma_Q$.
 \begin{theorem}[Happel-Ringel]\label{th:hr}\mbox{}
  \begin{enumerate}
   \item If $V$ and $W$ are indecomposable modules with $\Ext^1(V,W)=0$, then any nonzero map $W\to V$ is either a monomorphism or an epimorphism. 
   \item The dimension vectors of the indecomposable summands of a basic rigid module are linearly independent in the Grothendieck group of $\Gamma_Q$.
   \item A basic rigid module $T$ is a tilting module if and only if $T$ contains $n$ indecomposable summands.
  \end{enumerate}
 \end{theorem}
 Note that Theorem~\ref{th:hr}.3 implies each local tilting representation $T$ of $Q$ is a tilting representation for the full subquiver $Q^T$ of $Q$ where $Q^T_0=\supp(T)$.  We will consider the zero representation as a tilting representation for the empty subquiver.  A basic rigid representation is called an ``almost complete tilting representation'' if it contains $n-1$ indecomposable summands.  The following theorem describes the possible ways to complete an almost complete tilting representation to a tilting representation. 
 \begin{theorem}[Happel-Unger]\label{th:hu}
  Let $T$ be an almost complete tilting module.
  \begin{enumerate}
   \item If $T$ is sincere, then there exist exactly two non-isomorphic complements to $T$, otherwise there is a unique complement.
   \item Suppose $T$ is sincere and write $V$ and $W$ for the complements to $T$.  Suppose $\Ext^1(V,W)\ne0$, then there is a unique non-split sequence 
   \begin{tikzpicture}[baseline=-2pt]
    \matrix (m) [matrix of math nodes, row sep=3em, column sep=2.5em, text height=1.5ex, text depth=0.25ex]
     {0 & W & E & V & 0\ .\\};
    \path[->]
     (m-1-1) edge (m-1-2)
     (m-1-2) edge (m-1-3)
     (m-1-3) edge (m-1-4)
     (m-1-4) edge (m-1-5);
   \end{tikzpicture}
   Furthermore $E\in add(T)$ and $\dim_\FF\End(V)=\dim_\FF\End(W)=\dim_\FF\Ext^1(V,W)$.
  \end{enumerate}
 \end{theorem}

 Let $T$ be a local tilting representation of $(Q,\bfd)$.  We will use Theorem~\ref{th:hu} to define a mutation operation on $T$ which will produce another local tilting representation.  This construction first appeared in \cite{hub1}.  We will only consider those local tilting representations $T$ which may be obtained by iterated mutations starting from the zero representation.  Thus there will be a canonical labeling of the summands of $T$ by the vertices in its support and for $i\in\supp(T)$ we have $\End(T_i)\cong\End(S_i)\cong\FF_{d_i}$ by Theorem~\ref{th:hu}.2. 
 \begin{definition}
  Given a local tilting representation $T$ and a vertex $k\in Q_0$ we define the mutation $\mu_k(T)$ in direction $k$ as follows:
  \begin{enumerate}
   \item If $k\notin\supp T$, then by Theorem~\ref{th:hu}.1 there exists a unique complement $T_k^*$ so that $\mu_k(T)=T_k^*\oplus T$ is a local tilting representation containing $k$ in its support.
   \item If $k\in\supp T$, then write $\bar{T}=T/T_k$.
    \begin{enumerate}
     \item If $\bar{T}$ is a local tilting representation, i.e. $k\notin\supp\bar{T}$, let $\mu_k(T)=\bar{T}$.
     \item Otherwise $\supp\bar{T}=\supp T$ and by Theorem~\ref{th:hu} there exists a unique compliment $T_k^*\not\cong T_k$ so that $\mu_k(T)=T_k^*\oplus \bar{T}$ is a local tilting representation. 
    \end{enumerate}
  \end{enumerate}
 \end{definition}
 Notice that the definitions imply the mutation of local tilting representations is involutive.
 \begin{remark}
  It follows from results of \cite{bmrrt} that the mutation operation is transitive on the set of local tilting representations.
 \end{remark}

 \subsection{Matrices Associated to Local Tilting Representations}\label{sec:tilt_matrix} Here we define the matrix $B_T$ associated to a local tilting representation $T$.  We will show in Section~\ref{sec:exch_mut} that the matrix $B_T$ is skew-symmetrizable and that the mutation of local tilting representations induces the Fomin-Zelevinsky mutations of exchange matrices, see section~\ref{sec:qca}.  This construction was originally given in \cite{hub1}.  

 We will require a little preparation before we can define the matrix $B_T$.  For a valued representation $V$, we will call a morphism $V\to E$ (resp. $E\to V$) a ``left $add(\bar{T})$-approximation'' (resp. ``right $add(\bar{T})$-approximation'') of $V$ if
 \begin{itemize}
  \item $E\in add(\bar{T})$ and
  \item the induced map $\Hom(E,X)\to\Hom(V,X)$ (resp. $\Hom(X,E)\to\Hom(X,V)$) is surjective for any $X\in add(\bar{T})$.
 \end{itemize}
 In other words, $V\to E$ (resp. $E\to V$) is a left (resp. right) $add(\bar{T})$-approximation of $V$ if every map to (resp. from) an object of $add(\bar{T})$ factors through $V\to E$ (resp. $E\to V$).  A morphism $\phi:V\to E$ (resp. $\varphi:E\to V$) is called ``left minimal'' (resp. ``right minimal'') if any endomorphism $\psi$ of $E$ satisfying $\psi\circ\phi=\phi$ (resp. $\varphi\circ\psi=\varphi$) is an isomorphism.

 For each vertex $k$ of $Q$ our goal will be to define the entries $b_{ik}$ of the $k^{th}$ column of $B_T$.  First suppose $k$ is in the support of $T$ and, using the notation from the mutation of local tilting representations, suppose $\supp(\bar{T})=\supp(T)$.  Write $T_k^*$ for the other compliment of $\bar{T}$ described by the mutation.  Following Theorem~\ref{th:hu}.2 we assume that there is a unique non-split sequence
 \begin{equation}\label{eq:E_seq}
  \begin{tikzpicture}
   \matrix (m) [matrix of math nodes, row sep=1em, column sep=2.5em, text height=1.5ex, text depth=0.25ex]
    {0 & T_k^* & E & T_k & 0\,.\\};
   \path[->,font=\scriptsize]
    (m-1-1) edge (m-1-2)
    (m-1-2) edge (m-1-3)
    (m-1-3) edge (m-1-4)
    (m-1-4) edge (m-1-5);
  \end{tikzpicture}
 \end{equation}

 \begin{proposition}\label{prop:exch1}\cite[Lemma 21]{hub1}
  The map $T_k^*\to E$ is a minimal left $add(\bar{T})$-approximation of $T_k^*$ and the map $E\to T_k$ is a minimal right $add(\bar{T})$-approximation of $T_k$.
 \end{proposition}
 \begin{proof}
  By Theorem~\ref{th:hu}, $E\in add(\bar{T})$.  Let $X\in add(\bar{T})$ and note that $\Ext^1(T_k,X)=0$ since $T$ is rigid.  Thus applying $\Hom(-,X)$ to the sequence \eqref{eq:E_seq} gives the exact sequence
  \begin{center}
   \begin{tikzpicture}
    \matrix (m) [matrix of math nodes, row sep=3em, column sep=2.5em, text height=1.5ex, text depth=0.25ex]
     {\Hom(E,X) & \Hom(T_k^*,X) & 0\,,\\};
    \path[->]
     (m-1-1) edge (m-1-2)
     (m-1-2) edge (m-1-3);
   \end{tikzpicture}
  \end{center}
  and from the surjectivity of this map $E$ is a left $add(\bar{T})$-approximation of $T_k^*$.  The approximation $T_k^*\to E$ factors through any other approximation $T_k^*\to F$ and since $T_k^*\to E$ was injective we see that $T_k^*\to F$ must also be injective.  Define $G$ by the short exact sequence:
  \begin{equation}\label{eq:F_seq}
   \begin{tikzpicture}
    \matrix (m) [matrix of math nodes, row sep=1em, column sep=2.5em, text height=1.5ex, text depth=0.25ex]
     {0 & T_k^* & F & G & 0\,.\\};
    \path[->,font=\scriptsize]
     (m-1-1) edge (m-1-2)
     (m-1-2) edge (m-1-3)
     (m-1-3) edge (m-1-4)
     (m-1-4) edge (m-1-5);
   \end{tikzpicture}
  \end{equation}
  Since $F\in add(\bar{T})$ and $\bar{T}$ is rigid, applying $\Hom(\bar{T},-)$ to the sequence~\eqref{eq:F_seq} gives $\Ext^1(\bar{T},G)=0$.  Since $T_k^*\to F$ is a left $add(\bar{T})$-approximation, the map $\Hom(F,\bar{T})\to\Hom(T_k^*,\bar{T})$ is surjective and applying $\Hom(-,\bar{T})$ to \eqref{eq:F_seq} shows that $\Ext^1(G,\bar{T})=0$.  Then applying $\Hom(G,-)$ to \eqref{eq:F_seq} implies $G$ and hence $\bar{T}\oplus G$ are rigid.  So we must either have $G\in add(\bar{T}\oplus T_k)$ or $G\in add(\bar{T}\oplus T_k^*)$, but $F$ is a non-trivial extension in $\Ext^1(G,T_k^*)$ and thus $G\in add(T)$.  Moreover $G$ must contain $T_k$ as a summand.  But $E$ is the unique extension in $\Ext^1(T_k,T_k^*)$ and so the sequence~\eqref{eq:E_seq} is a summand of the sequence~\eqref{eq:F_seq}.   To see minimality of $E$, suppose we have an endomorphism $\psi$ of $E$ making the diagram
  \begin{center}
   \begin{tikzpicture}
    \matrix (m) [matrix of math nodes, row sep=1em, column sep=2.5em, text height=1.5ex, text depth=0.25ex]
     { & E \\ T_k^* & \\ & E \\};
    \path[->,font=\scriptsize]
     (m-2-1) edge (m-1-2) edge (m-3-2)
     (m-1-2) edge node[auto] {$\psi$} (m-3-2);
   \end{tikzpicture}
  \end{center}
  commute.  Then the image of $\psi$ is again an $add(\bar{T})$-approximation and from the discussion above we see that $E$ must be a summand, in other words $\psi$ is an isomorphism and $E$ is minimal.  The proof for $E\to T_k$ is similar.
 \end{proof}
Using the Auslander-Reiten formulas of Proposition~\ref{prop:AR}, the unique extension $E$ from equation~\ref{eq:E_seq} gives rise to a unique morphism $\theta\in\Hom(T_k^*,\tau T_k)$.  From $\theta$ we get a short exact sequence
 \begin{equation}\label{eq:theta-seq1}
  \begin{tikzpicture}
   \matrix (m) [matrix of math nodes, row sep=3em, column sep=2.5em, text height=1.5ex, text depth=0.25ex]
    {0 & D & T_k^* & \tau T_k & \tau A \oplus I & 0\\};
   \path[->,font=\scriptsize]
    (m-1-1) edge (m-1-2)
    (m-1-2) edge (m-1-3)
    (m-1-3) edge node[auto] {$\theta$}(m-1-4)
    (m-1-4) edge (m-1-5)
    (m-1-5) edge (m-1-6);
  \end{tikzpicture}
 \end{equation}
 where $D=\ker\theta$, $\tau A\oplus I=\coker\theta$, $I$ is injective, and $A$ and $T_k$ have the same maximal projective summand $P_A=P_{T_k}$.  As in the proof of Lemma~\ref{lem:hom-hall} the sequence \eqref{eq:theta-seq1} is equivalent to the following short exact sequences
 \begin{equation}\label{eq:AD_seq1}
  \begin{tikzpicture}
   \matrix (m) [matrix of math nodes, row sep=1em, column sep=2.5em, text height=1.5ex, text depth=0.25ex]
    {0 & D & T_k^* & C & 0\,,\\};
   \path[->,font=\scriptsize]
    (m-1-1) edge (m-1-2)
    (m-1-2) edge (m-1-3)
    (m-1-3) edge (m-1-4)
    (m-1-4) edge (m-1-5);
  \end{tikzpicture}
 \end{equation}
 \begin{equation}\label{eq:AD_seq2}
  \begin{tikzpicture}
   \matrix (m) [matrix of math nodes, row sep=1em, column sep=2.5em, text height=1.5ex, text depth=0.25ex]
    {0 & B & T_k & A & 0\,,\\};
   \path[->,font=\scriptsize]
    (m-1-1) edge (m-1-2)
    (m-1-2) edge (m-1-3)
    (m-1-3) edge (m-1-4)
    (m-1-4) edge (m-1-5);
  \end{tikzpicture}
 \end{equation}
 \begin{equation}\label{eq:AD_seq3}
  \begin{tikzpicture}
   \matrix (m) [matrix of math nodes, row sep=1em, column sep=2.5em, text height=1.5ex, text depth=0.25ex]
    {0 & C & \tau B & I & 0\,,\\};
   \path[->,font=\scriptsize]
    (m-1-1) edge (m-1-2)
    (m-1-2) edge (m-1-3)
    (m-1-3) edge (m-1-4)
    (m-1-4) edge (m-1-5);
  \end{tikzpicture}
 \end{equation}
 where $B$ contains no projective summands.  The following Proposition will allow us to complete the definition of the $k^{th}$ column of $B_T$ in the case under consideration and shows that we are in a position to apply Theorem~\ref{th:exchange_mult}.
 \begin{proposition}\label{prop:exch2}\cite[Proposition 22]{hub1}\mbox{}
  \begin{enumerate}
   \item The map $T_k\to A$ is a minimal left $add(\bar{T})$-approximation of $T_k$.
   \item The map $D\to T_k^*$ is a minimal right $add(\bar{T})$-approximation of $T_k^*$.
   \item The objects $B$ and $C$ are indecomposable.
   \item The objects $\soc I$ and $T$ have disjoint supports and $\Hom(A,I)=\Hom(D,I)=0$.
  \end{enumerate}
  Moreover, we have $\End(B)\cong\End(C)\cong\End(T_k)\cong\End(T_k^*)$.
 \end{proposition}
 \begin{proof}
  Any automorphism $\psi$ of the image $C$ of $\theta$ gives rise to another element $\psi\theta$ in $\Hom(T_k^*,\tau T_k)$.  The uniqueness of $\theta$ implies $C$ must be indecomposable.  Applying the functors $\Hom(\bar{T},-)$ and $\Hom(-,\bar{T})$ to the sequence \eqref{eq:AD_seq1} we see that $\Ext^1(\bar{T},C)=\Ext^1(D,\bar{T})=0$ and we get an exact sequence
  \begin{center}
   \begin{tikzpicture}
    \matrix (m) [matrix of math nodes, row sep=3em, column sep=2.5em, text height=1.5ex, text depth=0.25ex]
     {0 & \Hom(\bar{T},D) & \Hom(\bar{T},T_k^*) & \Hom(\bar{T},C) & \Ext^1(\bar{T},D) & 0\,.\\};
    \path[->]
     (m-1-1) edge (m-1-2)
     (m-1-2) edge (m-1-3)
     (m-1-3) edge (m-1-4)
     (m-1-4) edge (m-1-5)
     (m-1-5) edge (m-1-6);
   \end{tikzpicture}
  \end{center}
  Applying the same functors to the sequence \eqref{eq:AD_seq2} gives $\Ext^1(\bar{T},A)=\Ext^1(B,\bar{T})=0$ and an exact sequence
  \begin{center}
   \begin{tikzpicture}
    \matrix (m) [matrix of math nodes, row sep=3em, column sep=2.5em, text height=1.5ex, text depth=0.25ex]
     {0 & \Hom(A,\bar{T}) & \Hom(T_k,\bar{T}) & \Hom(B,\bar{T}) & \Ext^1(A,\bar{T}) & 0\,.\\};
    \path[->]
     (m-1-1) edge (m-1-2)
     (m-1-2) edge (m-1-3)
     (m-1-3) edge (m-1-4)
     (m-1-4) edge (m-1-5)
     (m-1-5) edge (m-1-6);
   \end{tikzpicture}
  \end{center}
  From the Auslander-Reiten formula we see that $\Hom(\bar{T},\tau B)=\Ext^1(B,\bar{T})=0$.  Then applying $\Hom(\bar{T},-)$ to the sequence \eqref{eq:AD_seq3} implies $\Hom(\bar{T},C)=\Hom(\bar{T},I)=0$ and again using the Auslander-Reiten formula we get $\Hom(B,\bar{T})=\Ext^1(\bar{T},\tau B)=\Ext^1(\bar{T},C)=0$.  Thus from the $\Hom$-sequences above we get $\Ext^1(\bar{T},D)=\Ext^1(A,\bar{T})=0$.  

  Applying $\Hom(-,D)$ to the sequences~\eqref{eq:E_seq}, where $E\in add(\bar{T})$, and \eqref{eq:AD_seq1} shows that $\Ext^1(T_k^*,D)=0$ and thus $\Ext^1(D,D)=0$.  Again using sequence~\eqref{eq:E_seq} and that $\Ext^1(\bar{T},C)=0$ we see $\Ext^1(T_k^*,C)=0$.  Then applying $\Hom(-,C)$ to \eqref{eq:AD_seq1} shows $\Ext^1(D,C)=0$ and finally applying $\Hom(D,-)$ to \eqref{eq:AD_seq1} gives $\Ext^1(D,T_k^*)=0$.  Thus we see that $D$ cannot contain $T_k$ as a summand and since \eqref{eq:AD_seq1} is non-split $D$ cannot contain $T_k^*$ as a summand.  We conclude that $D\in add(\bar{T})$.  A similar computation shows that $A\in add(\bar{T})$.

  Since $\Hom(B,\bar{T})=\Hom(\bar{T},C)=0$, the $\Hom$-sequences above imply that $T_k\to A$ is a left $add(\bar{T})$-approximation and $D\to T_k^*$ is a right $add(\bar{T})$-approximation.  As in the proof of Proposition~\ref{prop:exch1} the injectivity of $D\to T_k^*$ and the surjectivity of $T_k\to A$ imply they are minimal.

  Now since $D\in add(\bar{T})$ we have $\Hom(D,C)=0$ and applying $\Hom(-,C)$ to \eqref{eq:AD_seq1} gives $\Hom(C,C)=\Hom(T_k^*,C)$.  Since $T_k^*$ is indecomposable it cannot be the middle term of a split sequence.  But $\Ext^1(D,T_k^*)=0$ and so by Theorem~\ref{th:hr} any nonzero map from $T_k^*$ to a summand of $D$ must be surjective.  So we must have $\Hom(T_k^*,D)=0$ and applying $\Hom(T_k^*,-)$ again to the sequence \eqref{eq:AD_seq1} gives $\Hom(T_k^*,C)=\Hom(T_k^*,T_k^*)$.  Similarly one can show that $\Hom(B,A)=0$ implying $\Hom(B,B)=\Hom(B,T_k)$ and that $\Hom(A,T_k)=0$ so that $\Hom(B,T_k)=\Hom(T_k,T_k)$.  Then Theorem~\ref{th:hu} implies $\End(C)\cong\End(T_k^*)\cong\End(T_k)\cong\End(B)$.  

  Again we note that $\Ext^1(T_k^*,C)=0$ so that \eqref{eq:AD_seq3} implies $\Ext^1(T_k^*,\tau B)=0$.  Then using the Auslander-Reiten formula and \eqref{eq:AD_seq2} we see that $\Hom(T_k^*,\tau B)=\Ext^1(B,T_k^*)=\Ext^1(T_k,T_k^*)$, which according to Theorem~\ref{th:hu} is equal to $\Hom(T_k^*,T_k^*)=\Hom(T_k^*,C)$.  Thus
  \[\langle T_k^*,I\rangle = \langle T_k^*,\tau B\rangle-\langle T_k^*, C\rangle = \dim_\FF\Hom(T_k^*,\tau B) - \dim_\FF\Hom(T_k^*,C) = 0.\]
  Now $\langle -,I\rangle$ is zero on $add(T_k^*\oplus\bar{T})$, but by Theorem~\ref{th:hr} the dimension vectors of the indecomposable summands of $T_k^*\oplus\bar{T}$ form a basis of the Grothendieck group $\cK(Q_T)$.  Thus, since the Ringel-Euler form is non-degenerate, the support of $\soc I$ cannot intersect $\supp(T)$.

  Finally, since $A,D\in add(\bar{T})$ and $\Hom(\bar{T},I)=0$ we see that $\Hom(A,I)=\Hom(D,I)=0$ and thus the hypotheses of Theorem~\ref{th:exchange_mult} are satisfied.
 \end{proof}

 Following Theorem~\ref{th:hr} we know that the following subset of $\cK(Q)$ forms a basis: $\{[T_i]\}_{i\in\supp T}\cup\{[P_i]\}_{i\notin\supp T}$.  We will consider the $k^{th}$ column $\bfb^k$ of $B_T$ as an element of $\cK(Q)$ via $\bfb^k=\sum\limits_{i\in\supp T} b_{ik}[T_i]+\sum\limits_{i\notin\supp T} b_{ik}[P_i]$.  Then for $k\in\supp(T)=\supp(\bar{T})$ we may define the $k^{th}$ column of $B_T$ via
 \[\bfb^k=[A]+[D]+{}^{*_P}[I]-[E]\]
 where we write $\bfe^{*_P}=\sum\limits_{j\notin\supp T} \langle\bfe,\alpha_j^\vee\rangle[P_j]$ and ${}^{*_P}\bfe=\sum\limits_{j\notin\supp T} \langle\alpha_j^\vee,\bfe\rangle[P_j]$ for $\bfe\in\cK(Q)$.  The $k^{th}$ column of $B_{T_k^*\oplus\bar{T}}$ is by definition $-\bfb^k$.

 We prove the following consistency result which will be necessary to identify $B_T$ as an exchange matrix.
 \begin{lemma}\label{lem:h1}
  In the above basis, the elements $[A]+[D]$ and $[E]$ of $\cK(Q)$ have disjoint supports.
 \end{lemma}
 \begin{proof}
  We will show that the coefficients of $[T_i]$ in $[A]$ and $[E]$ cannot be simultaneously positive, the same argument will give the result for $[D]$ and $[E]$.  We argue for contradiction.  Suppose $[T_i]$ appears in both $[A]$ and $[E]$ with positive coefficient.  Then we have nonzero maps $\gamma:T_i\to T_k$ and $\zeta:T_k\to T_i$.  Since $\Ext^1(T_i,T_k)=\Ext^1(T_k,T_i)=0$, Theorem~\ref{th:hr} implies each of these maps is either a monomorphism or an epimorphism.  Notice that this implies one of the compositions $\gamma\circ\zeta$ or $\zeta\circ\gamma$ is nonzero.  Since $\End(T_i)\cong\FF_{d_i}$ and $\End(T_k)\cong\FF_{d_k}$, this implies the nonzero composition is an isomorphism.  But then both maps are injective and surjective, i.e. $T_i\cong T_k$, a contradiction.
 \end{proof}

 Now consider the situation where $\bar{T}$ is a local tilting representation with $k\notin\supp(\bar{T})$.  Write $T_k^*$ for the unique complement of $\bar{T}$.  Let $I_k$ denote the injective hull of the simple $S_k$ and write $P_k=\nu^{-1}(I_k)$ for the corresponding projective representation.  
 \begin{lemma}\label{lem:morphisms}\cite[Proposition 26]{hub1}
  There exists a unique nonzero morphism $P_k\to T_k^*$ and a unique morphism $T_k^*\to I_k$.
 \end{lemma}
 \begin{proof}
  Let $\bar{P}_k$ denote the projective representation of $Q_{T_k^*\oplus\bar{T}}$ associated to vertex $k$.  Note that there exists a unique morphism $P_k\to \bar{P}_k$ and since $T_k^*$ has support on $Q_{T_k^*\oplus\bar{T}}$, any morphism $P_k\to T_k^*$ factors through a morphism $\bar{P}_k\to T_k^*$.  Now recall that $k\notin\supp(\bar{T})$ and thus $\Hom(\bar{P}_k,\bar{T})=0$.  Since $T_k^*\oplus \bar{T}$ is a local tilting representation we have a coresolution 
  \begin{center}
   \begin{tikzpicture}
    \matrix (m) [matrix of math nodes, row sep=3em, column sep=2.5em, text height=1.5ex, text depth=0.25ex]
     {0 & \bar{P}_k & (T_k^*)^s & F & 0\,,\\};
    \path[->]
     (m-1-1) edge (m-1-2)
     (m-1-2) edge (m-1-3)
     (m-1-3) edge (m-1-4)
     (m-1-4) edge (m-1-5);
   \end{tikzpicture}
  \end{center}
  where $s\ge1$ and $F\in add(\bar{T})$.  Since $\Hom(\bar{P}_k,\bar{P}_k)$ is a field, applying the functor $\Hom(\bar{P}_k,-)$ to this coresolution shows $s=1$ and $\Hom(\bar{P}_k,\bar{P}_k)=\Hom(\bar{P}_k,T_k^*)$, in particular there is a unique (up to scalar) nonzero morphism $\bar{P}_k\to T_k^*$.  Composing with the unique morphism $P_k\to \bar{P}_k$ completes the claim.  The analogous claim for $I_k$ is dual.
 \end{proof}
 \begin{lemma}\label{lem:ends}\cite[Proposition 26]{hub1}
  The endomorphism rings $\End(P_k)$, $\End(\bar{P}_k)$, $\End(T_k^*)$, $\End(\bar{I}_k)$, and $\End(I_k)$ are all isomorphic.
 \end{lemma}
 \begin{proof}
  Since $C\in add(\bar{T})$ we know $\Ext^1(T_k^*,F)=0$ and according to Theorem~\ref{th:hr} any nonzero map from an indecomposable summand of $F$ to $T_k^*$ is injective.  But $T_k^*$ is indecomposable and so all such maps must be zero, in particular $\Hom(F,T_k^*)=0$.  Thus applying $\Hom(-,T_k^*)$ to the sequence 
  \begin{center}
   \begin{tikzpicture}
    \matrix (m) [matrix of math nodes, row sep=3em, column sep=2.5em, text height=1.5ex, text depth=0.25ex]
     {0 & \bar{P}_k & T_k^* & F & 0\\};
    \path[->]
     (m-1-1) edge (m-1-2)
     (m-1-2) edge (m-1-3)
     (m-1-3) edge (m-1-4)
     (m-1-4) edge (m-1-5);
   \end{tikzpicture}
  \end{center}
  gives $\Hom(\bar{P}_k,\bar{P}_k)=\Hom(\bar{P}_k,T_k^*)=\Hom(T_k^*,T_k^*)$.

  Since $P_k$ and $\bar{P}_k$ are both projective covers of the simple $S_k$, albeit as modules over different algebras, they have isomorphic endomorphism rings, i.e. $\End(P_k)\cong\End(S_k)\cong\End(\bar{P}_k)$.  The same argument holds for $I_k$ and $\bar{I}_k$ regarding them both as injective envelopes of $S_k$.
 \end{proof}

 As in the proof of Lemma~\ref{lem:morphisms}, write $F$ for the cokernel of the unique map $P_k\to T_k^*$ and let $P'$ denote the kernel.  Similarly write $G$ for the kernel of the unique morphism $T_k^*\to I_k$ and let $I'$ be the cokernel.  
 \begin{proposition}\label{prop:approx}\cite[Proposition 26]{hub1}
  The map $T_k^*\to F$ is a minimal left $add(\bar{T})$-approximation of $T_k^*$ and $G\to T_k^*$ is a minimal right $add(\bar{T})$-approximation of $T_k^*$.  Moreover, we have $\Hom(P',F)=\Hom(G,I')=0$ and the supports of $\soc I'$ and $P'/\rad P'$ are disjoint.
 \end{proposition}
 \begin{proof}
  Since $k\notin\supp(\bar{T})$, we have $\Hom(\bar{P}_k,\bar{T})=\Hom(\bar{T},\bar{I}_k)=0$.  Then applying $\Hom(-,\bar{T})$ and $\Hom(\bar{T},-)$ respectively to the defining exact sequences of $T_k^*\to F$ and $G\to T_k^*$ we see that $T_k^*\to F$ is a left $add(\bar{T})$-approximation and $G\to T_k^*$ is a right $add(\bar{T})$-approximation.  As in the proof of Proposition~\ref{prop:exch1}, these are minimal since $T_k^*\to F$ is surjective and $G\to T_k^*$ is injective.  

  Since $\bar{P}_k=P_k/P'$ and $I'=I_k/\bar{I}_k$, neither $P'$ nor $I'$ contains vertex $k$ in its support.  Recall that $P_k$ can be described in terms of paths beginning at vertex $k$ and $I_k$ can be described in terms of paths ending at vertex $k$.  Thus $Q$ being acyclic implies $[P'/\rad P']$ and $[\soc I']$ have disjoint support.

  Now since $k\notin\supp(\bar{T})$ we have $\Hom(P_k,\bar{T})=\Hom(P',\bar{T})=\Hom(P',F)=0$ and $\Hom(\bar{T},I_k)=\Hom(\bar{T},I')=\Hom(D,I')=0$. Thus the hypotheses of Theorem~\ref{th:proj_mult} are satisfied.
 \end{proof}

 As above we will consider the $k^{th}$ column $\bfb^k$ of $B_T$ as an element of $\cK(Q)$.  Now when $\bar{T}$ is a local tilting representation with $k\notin\supp(\bar{T})$ we may define the $k^{th}$ column of $B_T$ via
 \[\bfb^k=[F]+[P']^{*_P}-[G]-{}^{*_P}[I'].\]
 By definition the $k^{th}$ column of $B_{\bar{T}}$ is $-\bfb^k$.  Using a similar argument as in the proof of Lemma~\ref{lem:h1} we get the following consistency result.
 \begin{lemma}\label{lem:h2}
  Written in the basis $\{[T_i]\}_{i\in\supp T}\cup\{[P_i]\}_{i\notin\supp T}$, the elements $[F]$ and $[G]$ of $\cK(Q)$ have disjoint supports.
 \end{lemma}

 This completes the definition of the matrix $B_T$ associated to a local tilting representation $T$.  We will further investigate these matrices and their relationship to the mutation of local tilting representations in Section~\ref{sec:exch_mut} after we have presented the theory of quantum cluster algebras.

 \section{Quantum Cluster Algebras}\label{sec:qca}  In this section we will define quantum cluster algebras and recall some important structure theorems which motivate the main results of this article.

 We begin defining the quantum cluster algebra associated to an $m\times n$ ($m\ge n$) matrix $\tilde{B}$ with a skew-symmetrizable principal $n\times n$ submatrix $B$, i.e. there exists a diagonal matrix $D$ such that $DB$ is skew-symmetric.  An $m\times m$ skew-symmetric matrix $\Lambda$ is compatible with $\tilde{B}$ if 
 \[\tilde{B}^t\Lambda=\left(\begin{array}{cc}D & 0\end{array}\right).\]

 Let $q$ be an indeterminate.  The initial cluster $\bfX=\{X_1,\ldots,X_m\}$ of our quantum cluster algebra will form a generating set for an $m$-dimensional quantum torus 
  \[\cT_{\Lambda,q}=\ZZ[q^{\pm\half}]\langle X_1^{\pm1},\ldots,X_m^{\pm1}: X_iX_j=q^{\lambda_{ij}}X_jX_i\rangle.\]
 Note that $\cT_{\Lambda,q}$ is an Ore-domain and so it makes sense to consider its skew-field of fractions $\cF$.  The quantum cluster algebra will be a subalgebra of $\cF$.

 For $\bfa=(a_1,\ldots,a_m)\in\ZZ^m$ we will write
 \[X^\bfa=q^{-\half\sum\limits_{i<j}\lambda_{ij}a_ia_j}X_1^{a_1}\cdots X_m^{a_m}.\]
 Note that $X^\bfa$ is invariant under the unique anti-involution of $\cT_{\Lambda,q}$, called bar-involution, which fixes each $X_i$ and sends $q$ to $q^{-1}$.  Denoting by $\Lambda(\cdot,\cdot):\ZZ^m\times\ZZ^m\to \ZZ$ the skew-symmetric bilinear form associated to $\Lambda$ we have
 \[X^\bfa X^\bfb=q^{\half\Lambda(\bfa,\bfb)}X^{\bfa+\bfb}.\]
 We define a twisted multiplication $*:\cT_{\Lambda,q}\times\cT_{\Lambda,q}\to\cT_{\Lambda,q}$ by the rule
 \[X^\bfa*X^\bfb=X^{\bfa+\bfb}.\] 

 Let $e_1,\ldots,e_m$ be the standard basis vectors of $\ZZ^m$ and write $\bfa=\sum\limits_{i=1}^m a_ie_i\in\ZZ^n$.  For $1\le k\le n$ we will write $\bfb^k_+=\sum\limits_{i:b_{ik}>0} b_{ik}e_i$ and $\bfb^k_-=\sum\limits_{i:b_{ik}<0} -b_{ik}e_i$ for the positive and negative entries of the $k^{th}$ column of $\tilde{B}$, that is if we think of the $k^{th}$ column $\bfb^k$ of $\tilde{B}$ as an element of $\ZZ^m$, then $\bfb^k=\bfb^k_+-\bfb^k_-$.

 The collection $\Sigma=(\bfX,\tilde{B})$ is called a \emph{quantum seed}.  The Berenstein-Zelevinsky mutation $\mu_k$ of the seed $(\bfX,\tilde{B})$ in direction $k$ is given by $(\mu_k\bfX,\mu_k\tilde{B})$ where $\mu_k\bfX=\{X_1,\ldots,\hat{X_k},\ldots,X_m\}\cup\{X_k'\}$ for 
 \[X_k'=X^{\bfb^k_+-e_k}+X^{\bfb^k_--e_k},\]
 $\mu_k\tilde{B}$ is the Fomin-Zelevinsky \cite{fz2} mutation of exchange matrices given by $\mu_k\tilde{B}=(b'_{ij})$ where
 \begin{equation}\label{eq:exch_mutation}
  b'_{ij}=\begin{cases}-b_{ij} & \text{ if $i=k$ or $j=k$;}\\ b_{ij}+[b_{ik}]_+b_{kj}+b_{ik}[-b_{kj}]_+ & \text{ otherwise;}\end{cases}
 \end{equation}
 where $[b]_+=max(0,b)$, and $\mu_k\Lambda$ records the quasi-commutation of the cluster $\mu_k\bfX$, see \cite{berzel} for more details on mutations of compatible pairs.  We may also describe the mutation of $\tilde{B}$ in direction $k$ via $\mu_k\tilde{B}=E\tilde{B}F$ with $m\times m$ matrix $E=(e_{ij})$ and $n\times n$ matrix $F=(f_{ij})$ given by
 \begin{equation}\label{eq:EF}
  e_{ij}=\begin{cases} \delta_{ij} & \text{if $j\ne k$;}\\ -1 & \text{if $i=j=k$;}\\ [-b_{ik}]_+ & \text{if $i\ne j = k$;} \end{cases}\quad\quad\quad f_{ij}=\begin{cases} \delta_{ij} & \text{if $i\ne k$;}\\ -1 & \text{if $i=j=k$;}\\ [b_{kj}]_+ & \text{if $i=k\ne j$.} \end{cases}
 \end{equation}
 Then we may compute the commutation matrix $\mu_k\Lambda$ of the cluster $\mu_k\bfX$ as 
 \begin{equation}\label{eq:comm_mutation}
  \mu_k\Lambda=E^t\Lambda E.
 \end{equation}
 A result of Fomin and Zelevinsky \cite{fz4} asserts that the cluster variables of $\cA_q(\tilde{B},\Lambda)$ are completely determined by the cluster variables of the principal coefficients quantum cluster algebra $\cA(\tilde{B}_P,\Lambda')$ where $\tilde{B}_P=\left(\begin{array}{c} B\\ I\end{array}\right)$ with $I$ the $n\times n$ identity matrix and $\Lambda'$ is a compatible commutation matrix.  Thus we will always assume we are dealing with an exchange matrix having principal coefficients.

 \begin{definition}
  As with classical cluster algebras we may populate the vertices of an $n$-regular tree $\TT$ with quantum seeds where edges correspond to mutations.  The quantum cluster algebra $\cA_q(\tilde{B},\Lambda)$ is the $\ZZ[q^{\pm\half}]$-subalgebra of $\cF$ generated by the cluster variables from all seeds associated to vertices of $\TT$.
 \end{definition}

 The following Theorem of Berenstein and Zelevinsky \cite{berzel} is the first important structural result concerning quantum cluster algebras.
 \begin{theorem}[Quantum Laurent Phenomenon]\label{qlp}
  The quantum cluster algebra $\cA_q(\tilde{B},\Lambda)$ is a subalgebra of $\cT_{\Lambda,q}$.
 \end{theorem}
 \noindent This says that although the cluster variables a priori are rational functions, cancellations inevitably occur and we actually get Laurent polynomials.  A monomial in the cluster variables from a single quantum seed is called a \emph{quantum cluster monomial}.  Although the Quantum Laurent Phenomenon guarantees that each cluster monomial is an element of $\cT_{\Lambda,q}$ it is a non-trivial task to compute their initial cluster Laurent expansions. 

 We define a valued quiver $\tilde{Q}$ from a compatible pair $(\tilde{B},\Lambda)$ with principal coefficients as follows.  According to the construction in Section~\ref{sec:valued_quiver}. we may associated a valued quiver $(Q,\bfd)$ to the skew-symmetrizable principal $n\times n$ submatrix $B$ of $\tilde{B}$ where the valuation $d_i$ is the $i^{th}$ diagonal entry of the matrix $D$ occurring in the compatibility condition for $(\tilde{B},\Lambda)$.  Then we attach principal vertices $n+i\to i$ for each $1\le i\le n$ and set $d_{n+i}=d_i$ to obtain $\tilde{Q}$.  We will only be considering valued representations of $\tilde{Q}$ which are supported on $Q$ and thus we will only refer to the quiver $Q$ in the discussions that follow, however implicitly all equations/constructions involving the Grothendieck group $\cK(Q)$ or local tilting representations are happening inside $\Rep_\FF(\tilde{Q},\bfd)$.  The main result of this article is a description of the acyclic initial cluster expansion of all cluster monomials of $\cA(\tilde{B},\Lambda)$ using the representation theory of the valued quiver $(Q,\bfd)$.  We present the construction in the following section.

 \section{The Quantum Cluster Character $X_V$}\label{sec:qcc} In this section we study the quantum cluster character assigning an element of the quantum torus $\cT_{\Lambda,q}$ to each representation $V$ of $(Q,\bfd)$.  We will abbreviate $q=|\FF|$.

 \begin{definition}\label{def:qcc}
  For $V\in\Rep_\FF(Q,\bfd)$ define the \emph{quantum cluster character} $X_V$ in the quantum torus $\cT_{\Lambda,q}$ by
  \begin{equation}\label{eq:qcc}
   X_V=\sum\limits_{\bfe\in\cK(Q)} q^{-\half\langle\bfe,\bfv-\bfe\rangle}|Gr_\bfe^V|X^{-\bfe^*-{}^*(\bfv-\bfe)}
  \end{equation}
  where $Gr_\bfe^V=\{U\subset V:[U]=\bfe\}$ is the Grassmannian of all subobjects of $V$ with dimension vector $\bfe$.
 \end{definition}

 \subsection{Quantum Cluster Character Multiplication Theorems}\label{sec:mult_theorems}  In this section we prove multiplication formulas for products of quantum cluster characters.   Our first result is analogous to \cite[Theorem 12]{hub1}, \cite[Proposition 5.4.1]{qin}, and \cite[Theorem 3.5]{dingxu}.  All of their proofs are modeled on that of \cite[Theorem 12]{hub1} using Hall numbers, we will follow this approach as well. 

 Define the ``Hall number'' $F_{BC}^D$ as the number of subobjects $U\subset D$ with $U\cong C$ and $D/U\cong B$.  We also write $\varepsilon_{BC}^D$ for the size of the ``restricted $\Ext$-space'' $\Ext^1(B,C)_D\subset \Ext^1(B,C)$ consisting of those short exact sequences with middle term isomorphic to $D$.  Note that $F_{BC}^D$ and $\varepsilon_{BC}^D$ are finite since $B$, $C$, and $D$ are finite sets.  There are two well-known formulas satisfied by these quantities.  The first one verifies the associativity of the multiplication in a ``Hall algebra'' where the Hall numbers are structure constants.
 \begin{lemma}
  For any valued representations $B$, $K$, $L$, and $V$ we have the following ``associativity'' relation for Hall numbers:
  \begin{equation}\label{eq:hall-assoc}
   \sum\limits_A F_{KL}^AF_{AB}^V=\sum\limits_{A'} F_{KA'}^VF_{LB}^{A'}.
  \end{equation}
 \end{lemma}
 \noindent The second equation, known as ``Green's formula'', verifies the compatibility of multiplication and comultiplication of the Hall-Ringel algebra.  To state Green's formula we introduce the following useful notation:
 \[[V,W]^0:=\dim_\FF\Hom(V,W)\ \text{ and }\ [V,W]^1:=\dim_\FF\Ext^1(V,W),\ \text{ for all $V,W\in\Rep_\FF(Q,\bfd)$.} \]
 \begin{lemma} 
  For any valued representations $V$, $W$, $X$, and $Y$ 
  \begin{equation}\label{eq:green}
   \sum_E \varepsilon_{VW}^E F_{XY}^E=\sum\limits_{A,B,C,D} q^{[V,W]^0-[A,C]^0-[B,D]^0-\langle \bfa,\bfd\rangle} F_{AB}^V F_{CD}^W \varepsilon_{AC}^X \varepsilon_{BD}^Y.
  \end{equation}
 \end{lemma}

 Let $V$ and $W$ be representations of $(Q,\bfd)$.  From a morphism $\theta:W\to \tau V$ we get an exact sequence
 \begin{equation}\label{eq:theta-seq}
  \begin{tikzpicture}
   \matrix (m) [matrix of math nodes, row sep=3em, column sep=2.5em, text height=1.5ex, text depth=0.25ex]
    {0 & D & W & \tau V & \tau A \oplus I & 0\\};
   \path[->,font=\scriptsize]
    (m-1-1) edge (m-1-2)
    (m-1-2) edge (m-1-3)
    (m-1-3) edge node[auto] {$\theta$}(m-1-4)
    (m-1-4) edge (m-1-5)
    (m-1-5) edge (m-1-6);
  \end{tikzpicture}
 \end{equation}
 where $D=\ker\theta$, $\tau A\oplus I=\coker\theta$, $I$ is injective, and $A$ and $V$ have the same maximal projective summand $P_A=P_V$.  The following notation will be useful in the proof of the theorem:
 \[\Hom(W,\tau V)_{DAI}=\{W\xrightarrow{f} \tau V: f\ne0, \ker f\cong D, \coker f\cong \tau A\oplus I\}.\]
 The following Lemma was given in \cite{hub1}, we reproduce the proof for the convenience of the reader.  For a valued representation $C$ we write $a_C$ for the size of the automorphism group of $C$.
 \begin{lemma}\label{lem:hom-hall}\cite[Lemma 15]{hub1}
  The size of the restricted $\Hom$-space $\Hom(W,\tau V)_{DAI}$ can be computed in terms of Hall numbers, in particular we have
  \[|\Hom(W,\tau V)_{DAI}|=\sum\limits_{B,C} a_C F_{CD}^W F_{AB}^V F_{IC}^{\tau B},\]
  where $B$ contains no projective summands.
 \end{lemma}
 \begin{proof}
  Define 
  \[\cP_{XY}^Z=\{(s,t):
   \begin{tikzpicture}[baseline=-2pt]
    \matrix (m) [matrix of math nodes, row sep=3em, column sep=2.5em, text height=1.5ex, text depth=0.25ex]
     {0 & Y & Z & X & 0\\};
    \path[->,font=\scriptsize]
     (m-1-1) edge (m-1-2)
     (m-1-2) edge node[auto] {$s$} (m-1-3) 
     (m-1-3) edge node[auto] {$t$} (m-1-4) 
     (m-1-4) edge (m-1-5);
   \end{tikzpicture} \text{ is exact}\}\]
  and write $P_{XY}^Z=|\cP_{XY}^Z|$.  It is well known that the Hall number $F_{XY}^Z$ can be computed as $P_{XY}^Z/a_Xa_Y$.
  
  For any $\theta\in\Hom(W,\tau V)_{DAI}$ the exact sequence \eqref{eq:theta-seq} can be split into two short exact sequences:
  \begin{equation}\label{eq:theta-seq2}
   \begin{tikzpicture}
    \matrix (m) [matrix of math nodes, row sep=1em, column sep=2.5em, text height=1.5ex, text depth=0.25ex]
     {0 & D & W & C & 0,\\ 0 & C & \tau V & \tau A\oplus I & 0,\\};
    \path[->,font=\scriptsize]
     (m-1-1) edge (m-1-2)
     (m-1-2) edge node[auto] {$d$} (m-1-3)
     (m-1-3) edge node[auto] {$c$} (m-1-4)
     (m-1-4) edge (m-1-5)
     (m-2-1) edge (m-2-2)
     (m-2-2) edge node[auto] {$g$} (m-2-3)
     (m-2-3) edge node[auto] {$f$} (m-2-4)
     (m-2-4) edge (m-2-5);
   \end{tikzpicture}
  \end{equation}
  such that $gc=\theta$.  Thus we get a surjective map $\bigsqcup_C\cP_{CD}^W\times \cP_{\tau A\oplus I C}^{\tau V}\to \Hom(W,\tau V)_{DAI}$ with fiber over $\theta$ isomorphic to $\Aut(D)\times \Aut(C)\times \Aut(\tau A\oplus I)$.  Since the fibers are all isomorphic for a fixed $C$ we see that
  \begin{equation*}
   |\Hom(W,\tau V)_{DAI}|=\sum_C P_{CD}^W P_{\tau A\oplus I C}^{\tau V}/a_{\tau A\oplus I} a_C a_D=\sum_C a_C F_{CD}^W F_{\tau A\oplus I C}^{\tau V}.
  \end{equation*}
  From the surjective map $f:\tau V\to\tau A\oplus I$ we get a surjective map $\varphi:\tau V\to\tau A$.  Note that the kernel of this map contains no injective summands since $\tau A$ cannot have injective summands.  Thus we may write $\ker\varphi=\tau B$ for some $B$ containing no projective summands.  Since $B$ contains no projective summands, $\Hom(B,P_V)=0$ and applying the inverse Auslander-Reiten translation $\tau^{-1}$ induces a surjective map $V\to A$ with kernel $B$, where we have used that $P_V=P_A$.  Since $\im g=\ker f$, we must have $\im g\subset\ker\varphi$, in particular $g$ defines an injective map into $\tau B$.  The discussion thus far can be described by the commutative diagram:
  \begin{center}
   \begin{tikzpicture}
    \matrix (m) [matrix of math nodes, row sep=3em, column sep=2.5em, text height=1.5ex, text depth=0.25ex]
     {0 & C & \tau V & \tau A\oplus I & 0\\ 0 &\tau B & \tau V & \tau A & \ 0\,.\\};
    \path[->,font=\scriptsize]
     (m-1-1) edge (m-1-2)
     (m-1-2) edge (m-1-3) edge node[auto] {$g$} (m-2-2)
     (m-1-3) edge (m-1-4) edge node[auto] {$\Id$} (m-2-3)
     (m-1-4) edge (m-1-5) edge node[auto] {$\pi_1$} (m-2-4)
     (m-2-1) edge (m-2-2)
     (m-2-2) edge (m-2-3) 
     (m-2-3) edge node[auto] {$\varphi$} (m-2-4) 
     (m-2-4) edge (m-2-5);
   \end{tikzpicture}
  \end{center}
  Notice that $\tau A\oplus I\cong \tau V/C$ and $\tau A\cong \tau V/\tau B$ so that $I=\ker\pi_1\cong\tau B/C=\coker g$.  From this we see that the second sequence in \eqref{eq:theta-seq2} above is equivalent to the following exact sequences:
  \begin{center}
   \begin{tikzpicture}
    \matrix (m) [matrix of math nodes, row sep=1em, column sep=2.5em, text height=1.5ex, text depth=0.25ex]
     {0 & B & V & A & 0\,,\\ 0 & C & \tau B & I & 0\,,\\};
    \path[->,font=\scriptsize]
     (m-1-1) edge (m-1-2)
     (m-1-2) edge (m-1-3) 
     (m-1-3) edge (m-1-4) 
     (m-1-4) edge (m-1-5)
     (m-2-1) edge (m-2-2)
     (m-2-2) edge (m-2-3) 
     (m-2-3) edge (m-2-4) 
     (m-2-4) edge (m-2-5);
   \end{tikzpicture}
  \end{center}
  and so we have
  \[F_{\tau A\oplus I C}^{\tau V}=\sum_B F_{AB}^V F_{IC}^{\tau B}.\]
  The result follows.
 \end{proof}
 \noindent Note that by the Auslander-Reiten formula we have $\Ext^1(V,W)\cong\Hom(W,\tau V)$.
 \begin{theorem}\label{th:exchange_mult}
  Assume $V$ and $W$ are representations of $(Q,\bfd)$ with a unique (up to scalar) nontrivial extension $E\in\Ext^1(V,W)$, in particular $\dim_{\End(V)}\Ext^1(V,W)=1$.  Let $\theta\in\Hom(W,\tau V)$ be the equivalent morphism with $A,D,I$ as above.  Furthermore assume that $\Hom(A\oplus D,I)=0=\Ext^1(A,D)$.  Then we have the following multiplication formula:
  \begin{equation}\label{eq:mult1}
   X_VX_W=q^{\half\Lambda({}^*\bfv,{}^*\bfw)}X_E+q^{\half\Lambda({}^*\bfv,{}^*\bfw)+\half\langle\bfv,\bfw\rangle-\half\langle \bfa,\bfd\rangle}X_{D\oplus A}*X^{{}^*\bfi}.
  \end{equation}
 \end{theorem}
 \noindent When $\dim\Ext^1(V,W)>1$, there exist similar multiplication formulas with more terms, see \cite{fei} and \cite{dingsheng}.
 \begin{proof}
  Note that we have 
  \[|Gr_\bfe^V|=\sum\limits_{B,C:[C]=\bfe} F_{BC}^V\]
  and thus we may rewrite the quantum cluster character as
  \[X_V=\sum\limits_{A,B} q^{-\half\langle \bfb,\bfa\rangle} F_{AB}^V X^{-\bfb^*-{}^*\bfa}.\]
  Then using Lemma~\ref{lem:identities}.3 the product of the quantum cluster characters $X_V$ and $X_W$ becomes:
  \begin{align*}
   X_VX_W
   &=\sum\limits_{A,B}q^{-\half\langle\bfb,\bfa\rangle} F_{AB}^V X^{-\bfb^*-{}^*\bfa}\sum\limits_{C,D}q^{-\half\langle\bfd,\bfc\rangle} F_{CD}^W X^{-\bfd^*-{}^*\bfc}\\
   &=\sum\limits_{A,B,C,D} F_{AB}^V F_{CD}^W q^{-\half\langle\bfb,\bfa\rangle} q^{-\half\langle\bfd,\bfc\rangle} q^{\half\Lambda(-\bfb^*-{}^*\bfa,-\bfd^*-{}^*\bfc)} X^{-(\bfb+\bfd)^*-{}^*(\bfa+\bfc)}\\
   &=q^{\half\Lambda({}^*\bfv,{}^*\bfw)}\sum\limits_{A,B,C,D} F_{AB}^V F_{CD}^W q^{\langle\bfb,\bfc\rangle} q^{-\half\langle\bfb+\bfd,\bfa+\bfc\rangle} X^{-(\bfb+\bfd)^*-{}^*(\bfa+\bfc)}.
  \end{align*}
  Our goal is to show that this equals the right hand side of equation~\eqref{eq:mult1}.  We accomplish this by cleverly rewriting each term on the right.  We begin with the following definitions:
  \begin{align*}
   \sigma_1&:=\sum_{E\not\cong V\oplus W} \frac{\varepsilon_{VW}^E}{q^{[V,V]^0}-1} X_E,\\
   \sigma_2&:=\sum\limits_{\substack{D,A,I\\ D\not\cong W}} \frac{|\Hom(W,\tau V)_{DAI}|}{q^{[V,V]^0}-1} q^{-\half\langle \bfa,\bfd\rangle}X_{D\oplus A}*X^{{}^*\bfi}.
  \end{align*}
  Since there is a unique nontrivial extension $E\in\Ext^1(V,W)$ and a corresponding unique nonzero morphism $\theta\in\Hom(W,\tau V)$, we see that both of these sums collapse to a single term.  Since $\dim_{\End(V)}\Ext^1(V,W)=1$, we have $\varepsilon_{VW}^E=|\Hom(W,\tau V)_{DAI}|=q^{[V,V]^0}-1$ and thus the right hand side of equation~\eqref{eq:mult1} may be written as $q^{\half\Lambda({}^*\bfv,{}^*\bfw)}\sigma_1+q^{\half\Lambda({}^*\bfv,{}^*\bfw)+\half\langle\bfv,\bfw\rangle}\sigma_2$.  Observe that our computation of $X_VX_W$ above combined with the following Proposition complete the proof.

  \begin{proposition}\label{prop:sigma}
   We may rewrite $\sigma_1$ and $\sigma_2$ as
   \begin{align*}
    \sigma_1&=\sum\limits_{A,B,C,D} \frac{q^{[V,W]^1}-q^{[B,C]^1}}{q^{[V,V]^0}-1}q^{\langle \bfb,\bfc\rangle}q^{-\half\langle\bfb+\bfd,\bfa+\bfc\rangle}F_{AB}^V F_{CD}^W X^{-(\bfb+\bfd)^*-{}^*(\bfa+\bfc)},\\
    \sigma_2&=\sum\limits_{A,B,C,D} \frac{q^{[B,C]^1}-1}{q^{[V,V]^0}-1} F_{AB}^V F_{CD}^W q^{-\half\langle\bfv,\bfw\rangle}q^{\langle\bfb,\bfc\rangle}q^{-\half\langle\bfb+\bfd,\bfa+\bfc\rangle}X^{-(\bfb+\bfd)^*-{}^*(\bfa+\bfc)}.
   \end{align*}
  \end{proposition}
  \begin{proof}
   We begin with $\sigma_1$.  Using Green's formula \eqref{eq:green} we get
  \begin{align*}
   \sum_E \varepsilon_{VW}^E X_E
   &= \sum\limits_{E,X,Y} \varepsilon_{VW}^E q^{-\half\langle\bfy,\bfx\rangle} F_{XY}^E X^{-\bfy^*-{}^*\bfx}\\
   &= \sum\limits_{A,B,C,D,X,Y}q^{[V,W]^0-[A,C]^0-[B,D]^0-\langle \bfa,\bfd\rangle}q^{-\half\langle\bfb+\bfd,\bfa+\bfc\rangle}F_{AB}^V F_{CD}^W \varepsilon_{AC}^X \varepsilon_{BD}^Y X^{-\bfy^*-{}^*\bfx}\\
   &= \sum\limits_{A,B,C,D}q^{[V,W]^1}q^{\langle\bfb,\bfc\rangle}q^{-\half\langle\bfb+\bfd,\bfa+\bfc\rangle}F_{AB}^V F_{CD}^W X^{-(\bfb+\bfd)^*-{}^*(\bfa+\bfc)},
  \end{align*}
  where the last equality comes from the identities
  \begin{align*}
   \sum_X \varepsilon_{AC}^X&=|\Ext^1(A,C)|=q^{[A,C]^1},\\
   \sum_Y \varepsilon_{BD}^Y&=|\Ext^1(B,D)|=q^{[B,D]^1}.
  \end{align*}
  But the quantum cluster character gives
  \begin{align*}
   X_{V\oplus W}
   &=\sum\limits_{X,Y} q^{-\half\langle \bfy,\bfx\rangle} F_{XY}^{V\oplus W} X^{-\bfy^*-{}^*\bfx}\\
   &=\sum\limits_{A,B,C,D} q^{-\half\langle\bfb+\bfd,\bfa+\bfc\rangle} q^{[B,C]^0} F_{AB}^V F_{CD}^W X^{-(\bfb+\bfd)^*-{}^*(\bfa+\bfc)}.
  \end{align*}
  Since $\varepsilon_{VW}^{V\oplus W}=1$ we may rewrite $\sigma_1$ as 
  \begin{align*}
   \sigma_1
   &=\sum_{E\not\cong V\oplus W} \frac{\varepsilon_{VW}^E}{q^{[V,V]^0}-1} X_E\\
   &=\frac{\sum_E \varepsilon_{VW}^E X_E-X_{V\oplus W}}{q^{[V,V]^0}-1}\\
   &=\sum\limits_{A,B,C,D} \frac{q^{[V,W]^1}-q^{[B,C]^1}}{q^{[V,V]^0}-1}q^{\langle\bfb,\bfc\rangle}q^{-\half\langle\bfb+\bfd,\bfa+\bfc\rangle}F_{AB}^V F_{CD}^W X^{-(\bfb+\bfd)^*-{}^*(\bfa+\bfc)}.\\
  \end{align*}

  Now we move to $\sigma_2$.  Notice that by the Auslander-Reiten formula we have ${}^*\tau\bfb=-\bfb^*$.  And thus by Lemma~\ref{lem:hom-hall} we have 
   \begin{align*}
    \sigma_2
    &=\sum\limits_{\substack{A,D,I,K,L,X,Y\\ D\not\cong W}} \frac{|\Hom(W,\tau V)_{DAI}|}{q^{[V,V]^0}-1} q^{-\half\langle\bfa,\bfd\rangle} q^{[L,X]^0-\half\langle \bfl+\bfy,\bfk+\bfx\rangle} F_{KL}^A F_{XY}^D X^{-(\bfl+\bfy)^*-{}^*(\bfk+\bfx)+{}^*\bfi}\\
    &=\sum\limits_{\substack{A,B,C,D,I,K,L,X,Y\\ D\not\cong W}} \frac{a_C F_{CD}^W F_{AB}^V F_{IC}^{\tau B}}{q^{[V,V]^0}-1} q^{-\half\langle\bfa,\bfd\rangle} q^{[L,X]^0-\half\langle \bfl+\bfy,\bfk+\bfx\rangle} F_{KL}^A F_{XY}^D X^{-(\bfl+\bfy+\bfb)^*-{}^*(\bfk+\bfx+\bfc)}.
   \end{align*}
   By assumption $\Ext^1(A,D)=0$ and thus $\Ext^1(L,X)=0$.  So $\sigma_2$ becomes
   \begin{align*}
    \sigma_2
    &=\sum\limits_{\substack{A,B,C,D,I,K,L,X,Y\\ D\not\cong W}} \frac{a_C F_{CD}^W F_{AB}^V F_{IC}^{\tau B}}{q^{[V,V]^0}-1} q^{-\half\langle\bfa,\bfd\rangle} q^{\langle \bfl,\bfx\rangle-\half\langle \bfl+\bfy,\bfk+\bfx\rangle} F_{KL}^A F_{XY}^D X^{-(\bfl+\bfy+\bfb)^*-{}^*(\bfk+\bfx+\bfc)}.
   \end{align*}
   In the case $D=W$, we have $A=V$ and $C=B=I=0$.  Thus removing the condition $D\not\cong W$ we get
   \begin{align*}
    \sigma_2
    &=\sum\limits_{A,B,C,D,I,K,L,X,Y} \frac{a_C F_{CD}^W F_{AB}^V F_{IC}^{\tau B}}{q^{[V,V]^0}-1} q^{-\half\langle \bfa,\bfd\rangle}q^{\langle \bfl,\bfx\rangle-\half\langle \bfl+\bfy,\bfk+\bfx\rangle} F_{KL}^A F_{XY}^D X^{-(\bfl+\bfy+\bfb)^*-{}^*(\bfk+\bfx+\bfc)}\\
    &\quad-\sum\limits_{K,L,X,Y} \frac{1}{q^{[V,V]^0}-1} F_{KL}^V F_{XY}^W q^{-\half\langle \bfv,\bfw\rangle}q^{\langle \bfl,\bfx\rangle-\half\langle \bfl+\bfy,\bfk+\bfx\rangle} X^{-(\bfl+\bfy)^*-{}^*(\bfk+\bfx)}.
   \end{align*}
   Now we aim to remove the exponential dependence on $A$, $B$, $C$, $D$, $I$, $L$, and $X$ so that we may apply the associativity of Hall numbers and another simplifying equality.  To that end we claim the following identities:
   \begin{align*}
    \langle\bfk,\bfw\rangle+\langle\bfv,\bfk\rangle &= \langle\bfk,\bfd\rangle+\langle\bfa,\bfk\rangle,\\
    \langle\bfv,\bfy\rangle+\langle\bfy,\bfw\rangle &= \langle\bfa,\bfy\rangle+\langle\bfy,\bfd\rangle.
   \end{align*}
   Indeed, recall that by assumption $\Hom(D,I)=\Hom(A,I)=0$ and so $\Hom(K,I)=\Hom(Y,I)=0$.  Combining these observations with the Auslander-Reiten formulas of Proposition~\ref{prop:AR} gives
   \begin{align*}
    \langle\bfk,\bfw-\bfd\rangle &= \langle\bfk,\tau\bfv-\tau\bfa-\bfi\rangle = \langle\bfk,\tau\bfv-\tau\bfa\rangle = -\langle\bfv-\bfa,\bfk\rangle,\\
    \langle\bfy,\bfw-\bfd\rangle &= \langle\bfy,\tau\bfv-\tau\bfa-\bfi\rangle = \langle\bfy,\tau\bfv-\tau\bfa\rangle = -\langle\bfv-\bfa,\bfy\rangle.
   \end{align*}
   These now give the desired identity:
   \begin{align*}
    &-\half\langle\bfa,\bfd\rangle+\langle\bfl,\bfx\rangle-\half\langle\bfl+\bfy,\bfk+\bfx\rangle\\
    &\quad\quad= -\half\langle\bfa,\bfd\rangle+\langle\bfa-\bfk,\bfd-\bfy\rangle-\half\langle\bfa-\bfk+\bfy,\bfk+\bfd-\bfy\rangle\\
    &\quad\quad= -\half\langle\bfa,\bfy\rangle-\half\langle\bfy,\bfd\rangle-\half\langle\bfk,\bfd\rangle-\half\langle\bfa,\bfk\rangle+\half\langle\bfk,\bfy\rangle+\half\langle\bfk,\bfk\rangle-\half\langle\bfy,\bfk\rangle+\half\langle\bfy,\bfy\rangle\\
    &\quad\quad= -\half\langle\bfv,\bfy\rangle-\half\langle\bfy,\bfw\rangle-\half\langle\bfk,\bfw\rangle-\half\langle\bfv,\bfk\rangle+\half\langle\bfk,\bfy\rangle+\half\langle\bfk,\bfk\rangle-\half\langle\bfy,\bfk\rangle+\half\langle\bfy,\bfy\rangle\\
    &\quad\quad= \langle\bfv-\bfk,\bfw-\bfy\rangle-\half\langle\bfv-\bfk+\bfy,\bfk+\bfw-\bfy\rangle-\half\langle\bfv,\bfw\rangle.
   \end{align*}
   Now we may rewrite $\sigma_2$ as
   \begin{align*}
    \sigma_2
    &=\sum\limits_{A,B,C,D,I,K,L,X,Y} \frac{a_C F_{CD}^W F_{AB}^V F_{IC}^{\tau B}}{q^{[V,V]^0}-1} q^{\langle\bfv-\bfk,\bfw-\bfy\rangle-\half\langle\bfv-\bfk+\bfy,\bfk+\bfw-\bfy\rangle-\half\langle\bfv,\bfw\rangle} \times\\
    &\quad\quad\quad\quad\quad\quad\quad\quad\quad\quad\quad\quad\quad\quad\quad\quad\quad\quad\quad\times F_{KL}^A F_{XY}^D X^{-(\bfl+\bfy+\bfb)^*-{}^*(\bfk+\bfx+\bfc)}\\
    &\quad-\sum\limits_{K,L,X,Y} \frac{1}{q^{[V,V]^0}-1} F_{KL}^V F_{XY}^W q^{-\half\langle\bfv,\bfw\rangle}q^{\langle\bfl,\bfx\rangle-\half\langle \bfl+\bfy,\bfk+\bfx\rangle} X^{-(\bfl+\bfy)^*-{}^*(\bfk+\bfx)}
   \end{align*}
   where the dependence on $A$ only occurs in the product $F_{AB}^VF_{KL}^A$ and the $D$ dependence only occurs in the product $F_{CD}^WF_{XY}^D$.  In particular, the associativity of Hall numbers \eqref{eq:hall-assoc} applies and $\sigma_2$ may be rewritten as
   \begin{align*}
    \sigma_2
    &=\sum\limits_{A',B,C,D',I,K,L,X,Y} \frac{a_C F_{CX}^{D'} F_{LB}^{A'} F_{IC}^{\tau B}}{q^{[V,V]^0}-1} q^{\langle \bfv-\bfk,\bfw-\bfy\rangle-\half\langle \bfv-\bfk+\bfy,\bfk+\bfw-\bfy\rangle-\half\langle \bfv,\bfw\rangle} \times\\
    &\quad\quad\quad\quad\quad\quad\quad\quad\quad\quad\quad\quad\quad\quad\quad\quad\quad\quad\quad\quad\times F_{KA'}^V F_{D'Y}^W X^{-(\bfa'+\bfy)^*-{}^*(\bfk+\bfd')}\\
    &\quad-\sum\limits_{K,L,X,Y} \frac{1}{q^{[V,V]^0}-1} F_{KL}^V F_{XY}^W q^{-\half\langle\bfv,\bfw\rangle}q^{\langle \bfl,\bfx\rangle-\half\langle \bfl+\bfy,\bfk+\bfx\rangle} X^{-(\bfl+\bfy)^*-{}^*(\bfk+\bfx)}.
   \end{align*}
   By Lemma~\ref{lem:hom-hall} and the Auslander-Reiten formula~\eqref{eq:AR} we have
   \[\sum\limits_{I,L,X}\sum\limits_{B,C} a_C F_{CX}^{D'} F_{LB}^{A'} F_{IC}^{\tau B}=\sum\limits_{I,L,X}|\Hom(D',\tau A')_{XLI}| =q^{[D',\tau A']^0}=q^{[A',D']^1}.\]
   So that $\sigma_2$ becomes
   \begin{align*}
    \sigma_2 
    &=\sum\limits_{A',D',K,Y} \frac{q^{[A',D']^1}}{q^{[V,V]^0}-1} q^{\langle\bfv-\bfk,\bfw-\bfy\rangle-\half\langle\bfv-\bfk+\bfy,\bfk+\bfw-\bfy\rangle-\half\langle\bfv,\bfw\rangle} F_{KA'}^V F_{D'Y}^W X^{-(\bfa'+\bfy)^*-{}^*(\bfk+\bfd')}\\
    &\quad-\sum\limits_{K,L,X,Y} \frac{1}{q^{[V,V]^0}-1} F_{KL}^V F_{XY}^W q^{-\half\langle\bfv,\bfw\rangle}q^{\langle \bfl,\bfx\rangle-\half\langle \bfl+\bfy,\bfk+\bfx\rangle} X^{-(\bfl+\bfy)^*-{}^*(\bfk+\bfx)}\\
    &=\sum\limits_{A',D',K,Y} \frac{q^{[A',D']^1}}{q^{[V,V]^0}-1} F_{KA'}^V F_{D'Y}^W q^{-\half\langle\bfv,\bfw\rangle}q^{\langle\bfa',\bfd'\rangle-\half\langle \bfa'+\bfy,\bfk+\bfd'\rangle} X^{-(\bfa'+\bfy)^*-{}^*(\bfk+\bfd')}\\
    &\quad-\sum\limits_{K,L,X,Y} \frac{1}{q^{[V,V]^0}-1} F_{KL}^V F_{XY}^W q^{-\half\langle\bfv,\bfw\rangle}q^{\langle \bfl,\bfx\rangle-\half\langle \bfl+\bfy,\bfk+\bfx\rangle} X^{-(\bfl+\bfy)^*-{}^*(\bfk+\bfx)}\\
    &=\sum\limits_{A,B,C,D} \frac{q^{[B,C]^1}-1}{q^{[V,V]^0}-1} F_{AB}^V F_{CD}^W q^{-\half\langle\bfv,\bfw\rangle}q^{\langle\bfb,\bfc\rangle-\half\langle\bfb+\bfd,\bfa+\bfc\rangle} X^{-(\bfb+\bfd)^*-{}^*(\bfa+\bfc)}.
   \end{align*}
  \end{proof}
  \noindent This completes the proof of Proposition~\ref{prop:sigma} and thus of Theorem~\ref{th:exchange_mult}.
 \end{proof}

 Our final result of this section is analogous to \cite[Theorem 17]{hub1}, \cite[Proposition 5.4.1]{qin}, and \cite[Theorem 3.8]{dingxu}.  
 We again follow the Hall number approach of \cite[Theorem 17]{hub1}.  Let $W$ and $I$ be valued representations of $Q$, where $I$ is injective.  Write $P=\nu^{-1}(I)$ and note that $P$ is projective with $\soc I\cong P/\rad P$, moreover $\End(I)\cong\End(P)$.  From morphisms $\theta:W\to I$ and $\gamma:P\to W$ we get exact sequences
 \begin{center}
  \begin{tikzpicture}
   \matrix (m) [matrix of math nodes, row sep=1em, column sep=2.5em, text height=1.5ex, text depth=0.25ex]
    {0 & G & W & I & I' & 0\,,\\ 0 & P' & P & W & F & 0\,.\\};
   \path[->,font=\scriptsize]
    (m-1-1) edge (m-1-2)
    (m-1-2) edge (m-1-3)
    (m-1-3) edge node[auto] {$\theta$} (m-1-4)
    (m-1-4) edge (m-1-5)
    (m-1-5) edge (m-1-6)
    (m-2-1) edge (m-2-2)
    (m-2-2) edge (m-2-3)
    (m-2-3) edge node[auto] {$\gamma$} (m-2-4)
    (m-2-4) edge (m-2-5)
    (m-2-5) edge (m-2-6);
  \end{tikzpicture}
 \end{center}
 where $G=\ker\theta$, $I'=\coker\theta$ is injective, $P'=\ker\gamma$ is projective, and $F=\coker\gamma$.  We introduce the following notation needed for the proof of the theorem:
 \begin{align*}
  \Hom(W,I)_{GI'}&=\{W\xrightarrow{f}I:f\ne0, \ker f=G, \coker f=I'\},\\
  \Hom(P,W)_{P'F}&=\{P\xrightarrow{f}W:f\ne0, \ker f=P', \coker f=F\}.
 \end{align*}
 \begin{lemma}\label{lem:hom-hall2}
  The size of the restricted $\Hom$-spaces $\Hom(W,I)_{GI'}$ and $\Hom(P,W)_{P'F}$ can be computed in terms of Hall numbers, in particular we have
  \begin{align}
   \label{eq:hom-hall2.1}|\Hom(W,I)_{GI'}|&=\sum_A a_A F_{AG}^W F_{I'A}^I,\\
   \label{eq:hom-hall2.2}|\Hom(P,W)_{P'F}|&=\sum_B a_B F_{FB}^W F_{BP'}^P.
  \end{align}
 \end{lemma}
 \begin{proof}
  Notice that the $\theta$ exact sequence above can be split into the following two short exact sequences:
  \begin{center}
   \begin{tikzpicture}
    \matrix (m) [matrix of math nodes, row sep=1em, column sep=2.5em, text height=1.5ex, text depth=0.25ex]
     {0 & G & W & A & 0\,,\\ 0 & A & I & I' & 0\,,\\};
    \path[->,font=\scriptsize]
     (m-1-1) edge (m-1-2)
     (m-1-2) edge (m-1-3)
     (m-1-3) edge node[auto] {$c$} (m-1-4)
     (m-1-4) edge (m-1-5)
     (m-2-1) edge (m-2-2)
     (m-2-2) edge node[auto] {$g$} (m-2-3)
     (m-2-3) edge (m-2-4)
     (m-2-4) edge (m-2-5);
   \end{tikzpicture}
  \end{center}
  where $gc=\theta$.  Thus we have a surjective map $\bigsqcup_A\cP_{AG}^W\times\cP_{I'A}^I\to\Hom(W,I)_{GI'}$ with fiber over a morphism $\theta$ isomorphic to $\Aut(G)\times\Aut(A)\times\Aut(I')$.  Then the identity \eqref{eq:hom-hall2.1} follows from the equality
  \[|\Hom(W,I)_{GI'}|=\sum_A P_{AG}^WP_{I'A}^I/a_Ga_Aa_{I'}.\]

  Now notice the $\gamma$ exact sequence above can be split into the following two short exact sequences:
  \begin{center}
   \begin{tikzpicture}
    \matrix (m) [matrix of math nodes, row sep=1em, column sep=2.5em, text height=1.5ex, text depth=0.25ex]
     {0 & P' & P & B & 0\,,\\ 0 & B & W & F & 0\,,\\};
    \path[->,font=\scriptsize]
     (m-1-1) edge (m-1-2)
     (m-1-2) edge (m-1-3)
     (m-1-3) edge node[auto] {$d$} (m-1-4)
     (m-1-4) edge (m-1-5)
     (m-2-1) edge (m-2-2)
     (m-2-2) edge node[auto] {$h$} (m-2-3)
     (m-2-3) edge (m-2-4)
     (m-2-4) edge (m-2-5);
   \end{tikzpicture}
  \end{center}
  where $hd=\gamma$.  Thus we have a surjective map $\bigsqcup_B\cP_{BP'}^P\times\cP_{FB}^W\to\Hom(P,W)_{P'F}$ with fiber over a morphism $\gamma$ isomorphic to $\Aut(P')\times\Aut(B)\times\Aut(F)$.  Then the identity \eqref{eq:hom-hall2.2} follows from the equality
  \[|\Hom(P,W)_{P'F}|=\sum_B P_{BP'}^PP_{FB}^W/a_{P'}a_Ba_F.\]
 \end{proof}

 \begin{theorem}\label{th:proj_mult}
  Let $W$ and $I$ be valued representations of $Q$ with $I$ injective.  Write $P=\nu^{-1}(I)$.  Assume that there exist unique (up to scalar) morphisms $f\in\Hom(W,I)$ and $g\in\Hom(P,W)$, in particular $\dim_{\End(I)}\Hom(W,I)=\dim_{\End(P)}\Hom(P,W)=1$.  Define $F$, $G$, $I'$, $P'$ as above and assume further that $\Hom(P',F)=\Hom(G,I')=0$.  Then we have the following multiplication formula:
  \begin{align}\label{eq:proj_mult}
   X_WX^{{}^*\bfi}
   &=q^{-\half\Lambda({}^*\bfw,{}^*\bfi)}X_G*X^{{}^*\bfi'}+q^{-\half\Lambda({}^*\bfw,{}^*\bfi)-\half[I,I]^0}X_F*X^{{\bfp'}^*}.
  \end{align}
 \end{theorem}
 \begin{proof}
  We start by computing the product on the left using Lemma~\ref{lem:identities}.1:  
  \begin{align*}
   X_WX^{{}^*\bfi}
   &= \sum\limits_{X,Y} q^{-\half\langle\bfy,\bfx\rangle} F_{XY}^W q^{\half\Lambda(-\bfy^*-{}^*\bfx,{}^*\bfi)} X^{-\bfy^*-{}^*\bfx+{}^*\bfi}\\
   &= q^{-\half\Lambda({}^*\bfw,{}^*\bfi)}\sum\limits_{X,Y} q^{\half\Lambda({}^*\bfy-\bfy^*,{}^*\bfi)} q^{-\half\langle\bfy,\bfx\rangle} F_{XY}^W X^{-\bfy^*-{}^*\bfx+{}^*\bfi}\\
   &= q^{-\half\Lambda({}^*\bfw,{}^*\bfi)}\sum\limits_{X,Y} q^{-\half\langle \bfy, \bfi\rangle} q^{-\half\langle\bfy,\bfx\rangle} F_{XY}^W X^{-\bfy^*-{}^*\bfx+{}^*\bfi}.
  \end{align*} 
  Our goal is to show that this is equal to the right hand side of equation~\eqref{eq:proj_mult}.  We again accomplish this by cleverly rewriting each term on the right.  We make the following definitions 
  \begin{align*}
   \sigma_1
   &:=\sum\limits_{\substack{G,I'\\G\ne W}} \frac{|\Hom(W,I)_{GI'}|}{q^{[I,I]^0}-1}X_G*X^{{}^*\bfi'},\\
   \sigma_2
   &:=\sum\limits_{\substack{F,P'\\F\ne W}} \frac{|\Hom(P,W)_{P'F}|}{q^{[I,I]^0}-1}X_F*X^{{\bfp'}^*}.
  \end{align*}
  Since there are unique nonzero morphisms $W\to I$ and $P\to W$ each of these sums collapses to a single term.  Note that under the assumption $\dim_{\End(I)}\Hom(W,I)=\dim_{\End(P)}\Hom(P,W)=1$, we have $|\Hom(W,I)_{GI'}|=|\Hom(P,W)_{P'F}|=q^{[I,I]^0}-1$ where we have used that $\End(P)\cong\End(I)$.  Thus we see that the right hand side of equation~\eqref{eq:proj_mult} may be written as $q^{-\half\Lambda({}^*\bfw,{}^*\bfi)}\sigma_1+q^{-\half\Lambda({}^*\bfw,{}^*\bfi)-\half[I,I]^0}\sigma_2$.  Since $I$ is injective, applying $\Hom(-,I)$ to a short exact sequence 
  \begin{center}
   \begin{tikzpicture}
    \matrix (m) [matrix of math nodes, row sep=1em, column sep=2.5em, text height=1.5ex, text depth=0.25ex]
     {0 & Y & W & X & 0\\};
    \path[->,font=\scriptsize]
     (m-1-1) edge (m-1-2)
     (m-1-2) edge (m-1-3)
     (m-1-3) edge (m-1-4)
     (m-1-4) edge (m-1-5);
   \end{tikzpicture}
  \end{center}
  induces a short exact sequence
  \begin{center}
   \begin{tikzpicture}
    \matrix (m) [matrix of math nodes, row sep=1em, column sep=2.5em, text height=1.5ex, text depth=0.25ex]
     {0 & \Hom(Y,I) & \Hom(W,I) & \Hom(X,I) & 0\,.\\};
    \path[->,font=\scriptsize]
     (m-1-1) edge (m-1-2)
     (m-1-2) edge (m-1-3)
     (m-1-3) edge (m-1-4)
     (m-1-4) edge (m-1-5);
   \end{tikzpicture}
  \end{center}
  Since $\dim_{\End(I)}\Hom(W,I)=1$ we see that either $\Hom(Y,I)\cong\Hom(W,I)\cong\Hom(I,I)$ or $\Hom(Y,I)=0$.  Then $\langle \bfy,\bfi\rangle=[Y,I]_0$ either equals $[I,I]_0$ or $0$.  Now observe that our computation of $X_WX^{{}^*\bfi}$ above and the following Proposition complete the proof.
  \begin{proposition}\label{prop:sigma2}
   We may rewrite $\sigma_1$ and $\sigma_2$ as
   \begin{align*}
    \sigma_1&=\sum\limits_{X,Y:[Y,I]^0=0} q^{-\half\langle\bfy,\bfx\rangle} F_{XY}^W X^{-\bfy^*-{}^*\bfx+{}^*\bfi},\\
    \sigma_2&=\sum\limits_{X,Y:[Y,I]^0=[I,I]^0} q^{-\half\langle\bfy,\bfx\rangle} F_{XY}^W X^{-\bfy^*-{}^*\bfx+{}^*\bfi}.
   \end{align*}
  \end{proposition}
  \begin{proof}
   We begin with $\sigma_1$.  Using Lemma~\ref{lem:hom-hall2} we may rewrite $\sigma_1$ as 
   \begin{align*}
    \sigma_1
    &=\sum\limits_{\substack{A,G,I',X,Y\\G\ne W}} \frac{a_A F_{AG}^W F_{I'A}^I}{q^{[I,I]^0}-1} q^{-\half\langle\bfy,\bfx\rangle} F_{XY}^G X^{-\bfy^*-{}^*\bfx+{}^*\bfi'}\\
    &=\sum\limits_{\substack{A,G,I',X,Y\\G\ne W}} \frac{a_A F_{AG}^W F_{I'A}^I}{q^{[I,I]^0}-1} q^{-\half\langle\bfy,\bfx\rangle} F_{XY}^G X^{-\bfy^*-{}^*(\bfx+\bfa)+{}^*\bfi}.
   \end{align*}
   Note that by assumption we have $\Hom(G,I')=0$ and thus $\Hom(Y,I')=0$.  Since $G\ne W$, we have $A\ne0$ and $\Hom(A,I)\ne0$. Thus the induced exact diagram
   \begin{center}
    \begin{tikzpicture}
     \matrix (m) [matrix of math nodes, row sep=2.5em, column sep=2.5em, text height=1.5ex, text depth=0.25ex]
      {0 & \Hom(A,I) & \Hom(W,I) & \Hom(G,I) & 0\\ & & & \Hom(Y,I) & \\ & & & 0 & \\};
     \path[->,font=\scriptsize]
      (m-1-1) edge (m-1-2)
      (m-1-2) edge (m-1-3)
      (m-1-3) edge (m-1-4)
      (m-1-4) edge (m-1-5) edge (m-2-4)
      (m-2-4) edge (m-3-4);
    \end{tikzpicture}
   \end{center}
   implies $\Hom(Y,I)=\Hom(G,I)=0$.  So we get the identity
   \begin{align*}
    \langle\bfy,\bfx\rangle
    &= \langle\bfy,\bfw-\bfa-\bfy\rangle\\
    &= \langle\bfy,\bfw-\bfy\rangle-\langle\bfy,\bfi-\bfi'\rangle\\
    &= \langle\bfy,\bfw-\bfy\rangle,
   \end{align*}
   and $\sigma_1$ becomes
   \begin{align*}
    \sigma_1
    &=\sum\limits_{\substack{A,G,I',X,Y\\G\ne W}} \frac{a_A F_{AG}^W F_{I'A}^I}{q^{[I,I]^0}-1} q^{-\half\langle\bfy,\bfw-\bfy\rangle} F_{XY}^G X^{-\bfy^*-{}^*(\bfx+\bfa)+{}^*\bfi}.
   \end{align*}
   In the case $G=W$ we have $A=0$ and $I=I'$. Therefore we may rewrite $\sigma_1$ and then apply the associativity of Hall numbers as follows
   \begin{align*}
    \sigma_1
    &=\sum\limits_{A,G,I',X,Y} \frac{a_A F_{AG}^W F_{I'A}^I}{q^{[I,I]^0}-1} q^{-\half\langle\bfy,\bfw-\bfy\rangle} F_{XY}^G X^{-\bfy^*-{}^*(\bfx+\bfa)+{}^*\bfi}\\
    &\quad\quad-\sum\limits_{X,Y} \frac{1}{q^{[I,I]^0}-1} q^{-\half\langle\bfy,\bfw-\bfy\rangle} F_{XY}^W X^{-\bfy^*-{}^*\bfx+{}^*\bfi}\\
    &=\sum\limits_{A,G',I',X,Y} \frac{a_A F_{AX}^{G'} F_{I'A}^I}{q^{[I,I]^0}-1} q^{-\half\langle\bfy,\bfg'\rangle} F_{G'Y}^W X^{-\bfy^*-{}^*\bfg'+{}^*\bfi}\\
    &\quad\quad-\sum\limits_{X,Y} \frac{1}{q^{[I,I]^0}-1} q^{-\half\langle\bfy,\bfx\rangle} F_{XY}^W X^{-\bfy^*-{}^*\bfx+{}^*\bfi}.
   \end{align*}
   Notice that by Lemma~\ref{lem:hom-hall2} we have
   \[\sum\limits_{I',X}\sum\limits_A a_A F_{AX}^{G'} F_{I'A}^I=\sum\limits_{I',X}\Hom(G',I)_{XI'}=q^{[G',I]^0},\]
   so that $\sigma_1$ becomes
   \begin{align*}
    \sigma_1
    &=\sum\limits_{G',Y} \frac{q^{[G',I]^0}}{q^{[I,I]^0}-1} q^{-\half\langle\bfy,\bfg'\rangle} F_{G'Y}^W X^{-\bfy^*-{}^*\bfg'+{}^*\bfi}\\
    &\quad\quad-\sum\limits_{X,Y} \frac{1}{q^{[I,I]^0}-1} q^{-\half\langle\bfy,\bfx\rangle} F_{XY}^W X^{-\bfy^*-{}^*\bfx+{}^*\bfi}\\
    &=\sum\limits_{X,Y} \frac{q^{[X,I]^0}-1}{q^{[I,I]^0}-1} q^{-\half\langle\bfy,\bfx\rangle} F_{XY}^W X^{-\bfy^*-{}^*\bfx+{}^*\bfi}\\
    &=\sum\limits_{X,Y:[Y,I]^0=0} q^{-\half\langle\bfy,\bfx\rangle} F_{XY}^W X^{-\bfy^*-{}^*\bfx+{}^*\bfi}.\\
   \end{align*}

   Turning to $\sigma_2$, recall that we have $\bfp^*=[P/\rad P]=[\soc I]={}^*\bfi$.  Combining this observation with Lemma~\ref{lem:hom-hall2} we may rewrite $\sigma_2$ as 
   \begin{align*}
    \sigma_2
    &=\sum\limits_{\substack{B,F,P',X,Y\\F\ne W}} \frac{a_B F_{FB}^W F_{BP'}^P}{q^{[I,I]^0}-1} q^{-\half\langle\bfy,\bfx\rangle} F_{XY}^F X^{-\bfy^*-{}^*\bfx+{\bfp'}^*}\\
    &=\sum\limits_{\substack{B,F,P',X,Y\\F\ne W}} \frac{a_B F_{FB}^W F_{BP'}^P}{q^{[I,I]^0}-1} q^{-\half\langle\bfy,\bfx\rangle} F_{XY}^F X^{-(\bfy+\bfb)^*-{}^*\bfx+{}^*\bfi}.
   \end{align*}
   Note that by assumption we have $\Hom(P',F)=0$ and thus $\Hom(P',X)=0$.  Since $F\ne W$, we have $B\ne0$ and $\Hom(P,B)\ne0$.  Thus the induced exact diagram
   \begin{center}
    \begin{tikzpicture}
     \matrix (m) [matrix of math nodes, row sep=2.5em, column sep=2.5em, text height=1.5ex, text depth=0.25ex]
      {0 & \Hom(P,B) & \Hom(P,W) & \Hom(P,F) & 0\\ & & & \Hom(P,X) & \\ & & & 0 & \\};
     \path[->,font=\scriptsize]
      (m-1-1) edge (m-1-2)
      (m-1-2) edge (m-1-3)
      (m-1-3) edge (m-1-4)
      (m-1-4) edge (m-1-5) edge (m-2-4)
      (m-2-4) edge (m-3-4);
    \end{tikzpicture}
   \end{center}
   implies $\Hom(P,X)=\Hom(P,F)=0$.  So we get the identity
   \begin{align*}
    \langle\bfy,\bfx\rangle
    &= \langle\bfw-\bfb-\bfx,\bfx\rangle\\
    &= \langle\bfw-\bfx,\bfx\rangle-\langle\bfp-\bfp',\bfx\rangle\\
    &= \langle\bfw-\bfx,\bfx\rangle,
   \end{align*} 
   and $\sigma_2$ becomes
   \begin{align*}
    \sigma_2
    &=\sum\limits_{\substack{B,F,P',X,Y\\F\ne W}} \frac{a_B F_{FB}^W F_{BP'}^P}{q^{[I,I]^0}-1} q^{-\half\langle\bfw-\bfx,\bfx\rangle} F_{XY}^F X^{-(\bfy+\bfb)^*-{}^*\bfx+{}^*\bfi}.
   \end{align*}
   In the case $F=W$ we have $B=0$ and $P=P'$. Therefore we may rewrite $\sigma_2$ and then apply the associativity of Hall numbers as follows:
   \begin{align*}
    \sigma_2
    &=\sum\limits_{B,F,P',X,Y} \frac{a_B F_{FB}^W F_{BP'}^P}{q^{[I,I]^0}-1} q^{-\half\langle\bfw-\bfx,\bfx\rangle} F_{XY}^F X^{-(\bfy+\bfb)^*-{}^*\bfx+{}^*\bfi}\\
    &\quad\quad-\sum\limits_{X,Y} \frac{1}{q^{[I,I]^0}-1} q^{-\half\langle\bfw-\bfx,\bfx\rangle} F_{XY}^W X^{-\bfy^*-{}^*\bfx+{}^*\bfi}\\
    &=\sum\limits_{B,F',P',X,Y} \frac{a_B F_{YB}^{F'} F_{BP'}^P}{q^{[I,I]^0}-1} q^{-\half\langle\bff',\bfx\rangle} F_{XF'}^W X^{-{\bff'}^*-{}^*\bfx+{}^*\bfi}\\
    &\quad\quad-\sum\limits_{X,Y} \frac{1}{q^{[I,I]^0}-1} q^{-\half\langle\bfy,\bfx\rangle} F_{XY}^W X^{-\bfy^*-{}^*\bfx+{}^*\bfi}.
   \end{align*}
   Notice that by Lemma~\ref{lem:hom-hall2} we have
   \[\sum\limits_{P',Y}\sum\limits_B a_B F_{YB}^{F'}F_{BP'}^P=\sum\limits_{P',Y}\Hom(P,F')_{P'Y}=q^{[P,F']^0},\]
   so that $\sigma_2$ becomes
   \begin{align*}
    \sigma_2
    &=\sum\limits_{F',X} \frac{q^{[P,F']^0}}{q^{[I,I]^0}-1} q^{-\half\langle\bff',\bfx\rangle} F_{XF'}^W X^{-{\bff'}^*-{}^*\bfx+{}^*\bfi}\\
    &\quad\quad-\sum\limits_{X,Y} \frac{1}{q^{[I,I]^0}-1} q^{-\half\langle\bfy,\bfx\rangle} F_{XY}^W X^{-\bfy^*-{}^*\bfx+{}^*\bfi}\\
    &=\sum\limits_{X,Y} \frac{q^{[P,Y]^0}-1}{q^{[I,I]^0}-1} q^{-\half\langle\bfy,\bfx\rangle} F_{XY}^W X^{-\bfy^*-{}^*\bfx+{}^*\bfi}\\
    &=\sum\limits_{X,Y:[P,Y]^0=[I,I]^0} q^{-\half\langle\bfy,\bfx\rangle} F_{XY}^W X^{-\bfy^*-{}^*\bfx+{}^*\bfi}\\
    &=\sum\limits_{X,Y:[Y,I]^0=[I,I]^0} q^{-\half\langle\bfy,\bfx\rangle} F_{XY}^W X^{-\bfy^*-{}^*\bfx+{}^*\bfi}.
   \end{align*}
  \end{proof}
  \noindent This completes the proof of Proposition~\ref{prop:sigma2} and thus the proof of Theorem~\ref{th:proj_mult}.
 \end{proof}

 \subsection{Commutation and Compatibility}\label{sec:comm}
 Our eventual goal is to conclude that the quantum cluster character applied to exceptional representations of $(Q,\bfd)$ coincides with the initial cluster Laurent expansion of all non-initial cluster variables.  Recall that the cluster variables fit into quasi-commuting families and thus in this section we will consider what conditions we need on valued representations $V$ and $W$ so that the quantum cluster characters $X_V$ and $X_W$ quasi-commute.  The following Proposition, inspired by \cite[Prop. 3.6]{caldchap} and \cite[Equation (19)]{qin}, will be the main ingredient.

 \begin{proposition}\label{prop:grass_count}
  Let $V$ and $W$ be representations of $(Q,\bfd)$ with $\Ext^1(V,W)=0$.  Then
  \[|Gr_\bfe^{V\oplus W}|=\sum\limits_{\substack{\bfb,\bfc\in\cK(Q)\\\bfb+\bfc=\bfe}} q^{\langle\bfb,\bfw-\bfc\rangle}|Gr_\bfb^V||Gr_\bfc^W|.\]
 \end{proposition}
 \begin{proof}
  Denote by $\pi_1:V\oplus W\to V$ the natural projections and consider $W$ as a subrepresentations of $V\oplus W$ via the natural inclusion.  Fix $B\in Gr_\bfb^V$ and $C\in Gr_\bfc^W$ and define
  \[Gr_{B,C}^{V\oplus W}=\{L\in Gr_{\bfb+\bfc}^{V\oplus W}: \pi_1(L) = B, L\cap W = C\}.\]
  \begin{lemma}\label{lem:aff_fiber}
   The following map is an isomorphism:
   \begin{center}
    \begin{tikzpicture}
     \matrix (m) [matrix of math nodes, row sep=1em, column sep=2.5em, text height=1.5ex, text depth=0.25ex]
      {\zeta\ :\ \Hom(B,W/C) & Gr_{B,C}^{V\oplus W}\\ f & L^f:=\{b+w\in V\oplus W: b\in B, w\in W, f(b)=p(w)\}\\};
     \path[->]
      (m-1-1) edge (m-1-2);
     \path[|->]
      (m-2-1) edge (m-2-2);
    \end{tikzpicture}
   \end{center}
   where $p:W\to W/C$ is the natural projection.
  \end{lemma}
  \begin{proof}
   Since $p$ is surjective we can find for any $b\in B$ an element $w\in W$ so that $f(b)=p(w)$ and thus $\pi_1(L^f)=B$.  Also notice that $L^f\cap W=\text{ker }p=C$.  We define a map 
   \begin{center}
    \begin{tikzpicture}
     \matrix (m) [matrix of math nodes, row sep=1em, column sep=2.5em, text height=1.5ex, text depth=0.25ex]
      {\eta\ :\ Gr_{B,C}^{V\oplus W} & \Hom(B,W/C)\\ L & f^L\\};
     \path[->]
      (m-1-1) edge (m-1-2);
     \path[|->]
      (m-2-1) edge (m-2-2);
    \end{tikzpicture}
   \end{center}
   where $f^L(b):=p(w)$ for any $w\in W$ such that $b+w\in L$.  If $b+w\in L$ and $b+w'\in L$ then $w-w'\in L\cap W=C$ so that $p(w)=p(w')$ and $\eta$ is well defined.  It is easy to see that $\eta\circ\zeta$ and $\zeta\circ\eta$ are identity maps.
  \end{proof}
  \noindent Define a map 
   \begin{center}
    \begin{tikzpicture}
     \matrix (m) [matrix of math nodes, row sep=1em, column sep=2.5em, text height=1.5ex, text depth=0.25ex]
      {\varphi\ :\ Gr_\bfe^{V\oplus W} & \displaystyle\coprod\limits_{\bfb+\bfc=\bfe} Gr_\bfb^V\times Gr_\bfc^W\\ L & (\pi_1(L), L\cap W)\,.\\};
     \path[->]
      (m-1-1) edge (m-1-2);
     \path[|->]
      (m-2-1) edge (m-2-2);
    \end{tikzpicture}
   \end{center}
  The fiber over a point $(B,C)$ is $Gr_{B,C}^{V\oplus W}$, which by Lemma~\ref{lem:aff_fiber} is isomorphic to an affine space with $q^{\dim\Hom(B,W/C)}$ elements.  To complete the proof it suffices to show for $B\in Gr_\bfb^V$ and $C\in Gr_\bfc^W$ that $\dim\Hom(B,W/C)$ only depends on the dimension vectors of $B$ and $C$ and thus all fibers have the same number of points.  To this end we will show that $\Ext(B,W/C)=0$ so that $\dim\Hom(B,W/C)=\langle \bfb,\bfw-\bfc\rangle$.  Consider the following exact sequences:
  \begin{center}
   \begin{tikzpicture}
    \matrix (m) [matrix of math nodes, row sep=1em, column sep=2.5em, text height=1.5ex, text depth=0.25ex]
     {0 & B & V & V/B & 0\,,\\ 0 & C & W & W/C & 0\,.\\};
    \path[->,font=\scriptsize]
     (m-1-1) edge (m-1-2)
     (m-1-2) edge (m-1-3)
     (m-1-3) edge (m-1-4)
     (m-1-4) edge (m-1-5)
     (m-2-1) edge (m-2-2)
     (m-2-2) edge (m-2-3)
     (m-2-3) edge (m-2-4)
     (m-2-4) edge (m-2-5);
   \end{tikzpicture}
  \end{center}
  We apply $\Hom(-,W)$ to the first sequence and $\Hom(B,-)$ to the second sequence to get the following exact diagram taken from the corresponding long exact sequences:
  \begin{center}
   \begin{tikzpicture}
    \matrix (m) [matrix of math nodes, row sep=3em, column sep=2.5em, text height=1.5ex, text depth=0.25ex]
     { \Ext^1(V,W) & \Ext^1(B,W) & 0 \\ & \Ext^1(B,W/C) & \\ & 0 & \\};
    \path[->,font=\scriptsize]
     (m-1-1) edge (m-1-2)
     (m-1-2) edge (m-1-3) edge (m-2-2)
     (m-2-2) edge (m-3-2);
   \end{tikzpicture}
  \end{center}
  Since $\Ext^1(V,W)=0$ we get $\Ext^1(B,W/C)=0$.
 \end{proof}

 \begin{theorem}\label{th:qcc-comm}
  Let $V$ and $W$ be representations of $(Q,\bfd)$ with $\Ext^1(V,W)=0$, then 
  \[X_VX_W=q^{\half\Lambda({}^*\bfv,{}^*\bfw)}X_{V\oplus W}.\]
  If in addition $\Ext^1(W,V)=0$, we have
  \[X_VX_W = q^{\Lambda({}^*\bfv,{}^*\bfw)}X_WX_V.\]
 \end{theorem}
 \begin{proof}
  Using Lemma~\ref{lem:identities} and Proposition~\ref{prop:grass_count}, we have
  \begin{align*}
   X_VX_W&=\sum\limits_{\bfb,\bfc\in \cK(Q)} q^{-\half\langle\bfb,\bfv-\bfb\rangle} |Gr_\bfb^V| X^{-\bfb^*-{}^*(\bfc-\bfb)}\\
   & \quad \quad \cdot q^{-\half\langle\bfc,\bfw-\bfc\rangle} |Gr_\bfc^W| X^{-\bfc^*-{}^*(\bfw-\bfc)}\\
   &= \sum\limits_{\bfe\in \cK(Q)}\sum\limits_{\bfb+\bfc=\bfe} q^{\langle\bfb,\bfw-\bfc\rangle} |Gr_\bfb^V||Gr_\bfc^W| X^{-\bfe^*-{}^*(\bfv+\bfw-\bfe)}\\
   & \quad \quad \cdot q^{-\langle \bfb,\bfw-\bfc\rangle} q^{-\half\langle\bfb,\bfv-\bfb\rangle} q^{-\half\langle\bfc,\bfw-\bfc\rangle} q^{\half\Lambda(-\bfb^*-{}^*(\bfv-\bfb),-\bfc^*-{}^*(\bfw-\bfc))} \\
   &= q^{\half\Lambda({}^*\bfv,{}^*\bfw)}\sum\limits_{\bfe\in\cK(Q)}q^{-\half\langle \bfe,\bfv+\bfw-\bfe\rangle} |Gr_\bfe^{V\oplus W}| X^{-\bfe^*-{}^*(\bfv+\bfw-\bfe)}\\
   &= q^{\half\Lambda({}^*\bfv,{}^*\bfw)}X_{V\oplus W}.
  \end{align*}
 \end{proof}

 \begin{theorem}\label{th:init-comm}
  Let $V$ and $I$ be representations of $(Q,\bfd)$ such that $I$ is injective and $\supp(\soc I)\cap\supp(V)=\emptyset$, then
  \[X_VX^{{}^*\bfi} = q^{-\half\Lambda({}^*\bfv,{}^*\bfi)} X_V*X^{{}^*\bfi}\]
  and
  \[X_V X^{{}^*\bfi} = q^{-\Lambda({}^*\bfv,{}^*\bfi)} X^{{}^*\bfi} X_V.\]
 \end{theorem}
 \begin{proof}
  We compute the following products using Lemma~\ref{lem:identities}:
  \begin{align*}
   X_VX^{{}^*\bfi}
   &= \sum\limits_{\bfe\in\cK(Q)} q^{-\half\langle\bfe,\bfv-\bfe\rangle} |Gr_\bfe^V| q^{\half\Lambda(-\bfe^*-{}^*(\bfv-\bfe),{}^*\bfi)} X^{-\bfe^*-{}^*(\bfv-\bfe)+{}^*\bfi}\\
   &= q^{-\half\Lambda({}^*\bfv,{}^*\bfi)}\sum\limits_{\bfe\in\cK(Q)} q^{-\half\langle\bfe,\bfv-\bfe\rangle} |Gr_\bfe^V| q^{\half\Lambda({}^*\bfe-\bfe^*,{}^*\bfi)} X^{-\bfe^*-{}^*(\bfv-\bfe)+{}^*\bfi}\\
   &= q^{-\half\Lambda({}^*\bfv,{}^*\bfi)}\sum\limits_{\bfe\in\cK(Q)} q^{-\half\langle\bfe,\bfv-\bfe\rangle} |Gr_\bfe^V| q^{-\half\langle\bfe,\bfi\rangle} X^{-\bfe^*-{}^*(\bfv-\bfe)+{}^*\bfi},\\
   X^{{}^*\bfi}X_V
   &= \sum\limits_{\bfe\in\cK(Q)} q^{-\half\langle\bfe,\bfv-\bfe\rangle} |Gr_\bfe^V| q^{\half\Lambda({}^*\bfi,-\bfe^*-{}^*(\bfv-\bfe))} X^{-\bfe^*-{}^*(\bfv-\bfe)+{}^*\bfi}\\
   &= q^{-\half\Lambda({}^*\bfi,{}^*\bfv)}\sum\limits_{\bfe\in\cK(Q)} q^{-\half\langle\bfe,\bfv-\bfe\rangle} |Gr_\bfe^V| q^{\half\Lambda({}^*\bfi,{}^*\bfe-\bfe^*)} X^{-\bfe^*-{}^*(\bfv-\bfe)+{}^*\bfi}\\
   &= q^{-\half\Lambda({}^*\bfv,{}^*\bfi)}\sum\limits_{\bfe\in\cK(Q)} q^{-\half\langle\bfe,\bfv-\bfe\rangle} |Gr_\bfe^V| q^{\half\langle\bfe,\bfi\rangle} X^{-\bfe^*-{}^*(\bfv-\bfe)+{}^*\bfi}.
  \end{align*}
  For any $i\in\supp(\soc I)$ there are no morphisms from $V$ to $S_i$ and thus $\Hom(V,I)=0$.  This implies $\langle\bfe,\bfi\rangle=\langle\bfv,\bfi\rangle=0$ and the claim follows.  
 \end{proof}
 \begin{remark}\label{rem:init_cluster_comm}
  It is clear that for any injective valued representations $I$ and $J$, we have
  \[X^{{}^*\bfi}X^{{}^*\bfj}=q^{\Lambda({}^*\bfi,{}^*\bfj)}X^{{}^*\bfj}X^{{}^*\bfi}.\]
  In particular, if $I=I_i$ and $J=I_j$ are the injective hulls of the simples $S_i$ and $S_j$, respectively, then we have 
  \[X^{{}^*\bfi}X^{{}^*\bfj}=q^{\lambda_{ij}}X^{{}^*\bfj}X^{{}^*\bfi}.\]
 \end{remark}

 We see from Theorem~\ref{th:qcc-comm} that a family of valued representations $V_1,\ldots,V_k$ will have quasi-commuting quantum cluster characters $X_{V_1},\ldots,X_{V_k}$ exactly when $V_1\oplus\cdots\oplus V_k$ is rigid.  Moreover, Theorem~\ref{th:qcc-comm} implies that for any rigid decomposable valued representation $V$ we may factor the quantum cluster character $X_V$.  This suggests that we should further restrict to indecomposable $V_i$ such that $V_i\not\cong V_j$ for $i\ne j$.  Since they both give rise to quasi-commuting families we would naturally suspect that there should be a relationship between clusters of $\cA_q(\tilde{B},\Lambda)$ and basic rigid representations of $(Q,\bfd)$.  Now the support condition satisfied by local tilting representations combined with Theorem~\ref{th:init-comm} is exactly what is needed to guarantee that we can obtain in this way a full cluster of $n$ mutable variables.  We will make these remarks precise in Section~\ref{sec:quantum_seed} when we explicitly construct a seed of $\cA_q(\tilde{B},\Lambda)$ from each local tilting representation of $(Q,\bfd)$.

 \section{Mutations of Exchange Matrices}\label{sec:exch_mut} Here we show that the Fomin-Zelevinsky mutation of exchange matrices associated to local tilting representations coincides with the mutation operation for local tilting representations defined in Section~\ref{sec:tilt_matrix}.

 Since the valued quiver $Q$ has no oriented cycles, according to \cite[Section 2.4]{ringel2} the Ringel-Euler form is non-degenerate.  Also by Theorem~\ref{th:hr} the classes of indecomposable summands of a local tilting representation $T$ are linearly independent in the Grothendieck group $\cK(Q)$.  Since there are as many non-isomorphic summands of $T$ as simple modules for $Q_T$, their isomorphism classes form a basis of the Grothendieck group $\cK(Q_T)\subset\cK(Q)$.  Notice that the isomorphism classes of projective objects $P_i$ for $i\notin\supp T$ are linearly independent in $\cK(Q)$ and the Ringel-Euler form is non-degenerate when restricted to their span.  Moreover, if we consider the isomorphism classes of summands of $T$ and these projective representations we obtain a basis for all of $\cK(Q)$.  Define the natural projections $\pi_T:\cK(Q)\to\cK(Q_T)$ and $\pi_T^c:\cK(Q)\to \spa\{[P_i]:i\notin\supp T\}$.
 \begin{definition}\mbox{}
  \begin{enumerate}
   \item Since the Ringel-Euler form is non-degenerate on $\cK(Q_T)$, for $i\in\supp T$ we may define left and right duals $\lambda_i,\rho_i\in\cK(Q)$ and $\lambda_i^T, \rho_i^T$ in $\cK(Q_T)$ satisfying: 
   \begin{align*}
    \pi_T(\lambda_i)=\lambda_i^T &\text{ and } \pi_T(\rho_i)=\rho_i^T,\\
    \langle\lambda_i^T,T_j\rangle=\delta_{ij}d_i &\text{ and } \langle T_j,\rho_i^T\rangle=\delta_{ij}d_i \text{ for all } j\in\supp T.
   \end{align*}
   \item For $i\notin\supp T$ we define left and right duals $\lambda_i,\rho_i\in\cK(Q)$, $\lambda_i^T,\rho_i^T\in\cK(Q_T)$, and $\lambda_i^c,\rho_i^c\in\spa\{P_i:i\notin\supp T\}$ satisfying the following identities:
   \begin{align*}
    \pi_T(\lambda_i)=\lambda_i^T &\text{ and } \pi_T(\rho_i)=\rho_i^T,\\
    \pi_T^c(\lambda_i)=\lambda_i^c &\text{ and } \pi_T^c(\rho_i)=\rho_i^c,\\
    \langle\lambda_i^T,T_j\rangle=\langle S_i,T_j\rangle &\text{ and } \langle T_j,\rho_i^T\rangle=\langle T_j,S_i\rangle \text{ for all } j\in\supp T,\\
    \langle\lambda_i^c,P_j\rangle=\delta_{ij}d_i &\text{ and } \langle P_j,\rho_i^c\rangle=\delta_{ij}d_i \text{ for all } j\notin\supp T.
   \end{align*}
  \end{enumerate} 
 \end{definition}
 \noindent Write $\lambda_i=\sum\limits_{j\in\supp T} \ell_{ij}[T_j]+\sum\limits_{j\notin\supp T} \ell_{ij}[P_j]$ and $\rho_i=\sum\limits_{j\in\supp T} r_{ij}[T_j]+\sum\limits_{j\notin\supp T} r_{ij}[P_j]$.  Note that it still remains to define $\ell_{ij}$ and $r_{ij}$ for $i\in\supp T$ and $j\notin\supp T$.  We begin with the following Lemma which motivates that definition.
 \begin{lemma}\mbox{}
  \begin{enumerate}
   \item The matrices $L^T=(\ell_{ij})$ and $R^T=(r_{ij})$ with rows and columns labeled by $\supp T$ are related by $L^TD^T=(D^TR^T)^t$ where $D^T=(d_{ij})$ is the diagonal matrix given by $d_{ij}=\delta_{ij}d_i$ for $i,j\in\supp T$.
   \item The matrices $L^c=(\ell_{ij})$ and $R^c=(r_{ij})$ with rows and columns labeled by $i,j\notin\supp T$ are related by $L^cD^c=(D^cR^c)^t$ where $D^c=(d_{ij})$ is the diagonal matrix given by $d_{ij}=\delta_{ij}d_i$ for $i,j\notin\supp T$. 
  \end{enumerate}
 \end{lemma}
 \begin{proof}
  By pairing $\lambda_i^T$ and $\rho_j^T$ for $i,j\in\supp T$ we see that $\ell_{ij}^Td_j=r_{ji}^Td_i$.  Similarly, by pairing $\lambda_i^c$ and $\rho_j^c$ for $i,j\notin\supp T$ we see that $\ell_{ij}^cd_j=r_{ji}^cd_i$.  The result follows.
 \end{proof}
 \noindent Thus we complete the definition of the matrices $L=(\ell_{ij})$ and $R=(r_{ij})$ by declaring that $LD=(DR)^t$.  Note that this uniquely defines $\pi_T^c(\lambda_i)$ and $\pi_T^c(\rho_i)$ for $i\in\supp T$.
 
 Suppose $k\in\supp(T)=\supp(\bar{T})$.  We will write $\lambda_i^{T_k}$, $\rho_i^{T_k}$, $\lambda_i^{T_k^*}$ and $\rho_i^{T_k^*}$ for the duals with respect to the $T_k\oplus\bar{T}$-basis and the $T_k^*\oplus\bar{T}$-basis respectively.  Recall the matrix $B_T=(b_{ij}^{T_k})$ from Section~\ref{sec:tilt_matrix}.
 \noindent The following Lemma is based upon \cite[Lemmas 24 \& 25]{hub1}.
 \begin{lemma}\label{lem:24}
  We may compute $\pi_T(\lambda_k^{T_k})$, $\pi_T(\lambda_k^{T_k^*})$, $\pi_T(\rho_k^{T_k})$, and $\pi_T(\rho_k^{T_k^*})$ as follows:
  \begin{equation*}
   \pi_T(\lambda_k^{T_k})=-\pi_T(\lambda_k^{T_k^*})=[T_k]-[A] \text{ and } -\pi_T(\rho_k^{T_k})=\pi_T(\rho_k^{T_k^*})=[T_k^*]-[D].
  \end{equation*}
  We may also compute $\pi_T^c(\lambda_k^{T_k})$, $\pi_T^c(\lambda_k^{T_k^*})$, $\pi_T^c(\rho_k^{T_k})$, and $\pi_T^c(\rho_k^{T_k^*})$ as follows:
  \begin{equation*}
   \pi_T^c(\lambda_k^{T_k})=-\pi_T^c(\lambda_k^{T_k^*})=[B]^{*_P} \text{ and } \pi_T^c(\rho_k^{T_k})=-\pi_T^c(\rho_k^{T_k^*})=-{}^{*_P}[C],
  \end{equation*}
  where we write $\bfe^{*_P}=\sum\limits_{j\notin\supp T} \langle\bfe,\alpha_j^\vee\rangle[P_j]$ and ${}^{*_P}\bfe=\sum\limits_{j\notin\supp T} \langle\alpha_j^\vee,\bfe\rangle[P_j]$.  Moreover, for $i\in\supp T$ we have
  \begin{align*}
   \lambda_i^{T_k^*}=\lambda_i^{T_k}+[-b_{ik}^{T_k}]_+d_i/d_k\lambda_k^{T_k} &\text{ and } \rho_i^{T_k^*}=\rho_i^{T_k}+[-b_{ik}^{T_k}]_+d_i/d_k\rho_k^{T_k},
  \end{align*}
  and for $i\notin\supp T$ we have
  \begin{align*}
   \lambda_i^{T_k^*}=\lambda_i^{T_k} &\text{ and } \rho_i^{T_k^*}=\rho_i^{T_k}.
  \end{align*}
 \end{lemma}
 \begin{proof}
  According to the proof of Proposition~\ref{prop:exch2} we have $\langle A, T_k\rangle=0$ and $\langle A,T_j\rangle=\langle T_k,T_j\rangle$ for all $j\in\supp(T)$, $j\ne k$.  Since $A\in add(\bar{T})$ we see that $\pi_T(\lambda_k^{T_k})=[T_k]-[A]$.  Similarly one sees that $\pi_T(\rho_k^{T_k^*})=[T_k^*]-[D]$.  Now notice that
  \begin{align*}
   &\langle\pi_T(\lambda_k^{T_k})+\pi_T(\lambda_k^{T_k^*}),[T_j]\rangle=0 \text{ and }\\ 
   &\langle\pi_T(\lambda_k^{T_k})+\pi_T(\lambda_k^{T_k^*}),[T_k]\rangle=d_k+\langle\pi_T(\lambda_k^{T_k^*}),[E]-[T_k^*]\rangle=d_k-d_k=0,
  \end{align*}
  so that $\pi_T(\lambda_k^{T_k})+\pi_T(\lambda_k^{T_k^*})=0$ by the non-degeneracy of the Ringel-Euler form on $\cK(Q_T)$.  Similarly we compute
  \begin{align*}
   &\langle[T_j],\pi_T(\rho_k^{T_k})+\pi_T(\rho_k^{T_k^*})\rangle=0 \text{ and }\\ 
   &\langle[T_k],\pi_T(\rho_k^{T_k})+\pi_T(\rho_k^{T_k^*})\rangle=d_k+\langle[E]-[T_k^*],\pi_T(\rho_k^{T_k^*})\rangle=d_k-d_k=0,
  \end{align*}
  so that $\pi_T(\rho_k^{T_k})+\pi_T(\rho_k^{T_k^*})=0$.  Assume $i,j\in\supp T$ and $i,j\ne k$.  Then we may obtain identities relating $\pi_T(\lambda_i^{T_k})$ and $\pi_T(\lambda_i^{T_k^*})$ as follows:
  \begin{align*}
   &\langle\pi_T(\lambda_i^{T_k})-\pi_T(\lambda_i^{T_k^*}),[T_j]\rangle=0 \text{ and }\\ 
   &\langle\pi_T(\lambda_i^{T_k})-\pi_T(\lambda_i^{T_k^*}),[T_k]\rangle=-\langle\pi_T(\lambda_i^{T_k^*}),[E]-[T_k^*]\rangle=-[-b_{ik}^{T_k}]_+d_i,
  \end{align*}
  so that $\pi_T(\lambda_i^{T_k})-\pi_T(\lambda_i^{T_k^*})=-[-b_{ik}^{T_k}]_+d_i/d_k\pi_T(\lambda_k^{T_k})$.  We remark that this implies the same identity holds when we replace $\pi_T$ by $\pi_T^c$.  Indeed, for $j\notin\supp T$ we may compute the pairing $\langle-,\pi_T^c(\rho_j^{T_k})\rangle$ on both sides of the desired identity as follows:
  \begin{align*}
   \langle\pi_T^c(\lambda_i^{T_k})-\pi_T^c(\lambda_i^{T_k^*}),\pi_T^c(\rho_j^{T_k})\rangle
   &=\ell_{ij}^{T_k}d_j-\ell_{ij}^{T_k^*}d_j\\
   &=r_{ji}^{T_k}d_i-r_{ji}^{T_k^*}d_i\\
   &=\langle\pi_T(\lambda_i^{T_k})-\pi_T(\lambda_i^{T_k^*}),\pi_T(\rho_j^{T_k})\rangle,\\ 
   \langle-[-b_{ik}^{T_k}]_+d_i/d_k\pi_T^c(\lambda_k^{T_k}),\pi_T^c(\rho_j^{T_k})\rangle
   &=-[-b_{ik}^{T_k}]_+d_i/d_k\ell_{kj}^{T_k}d_j\\
   &=-[-b_{ik}^{T_k}]_+d_i/d_kr_{jk}^{T_k}d_k\\
   &=\langle-[-b_{ik}^{T_k}]_+d_i/d_k\pi_T(\lambda_k^{T_k}),\pi_T(\rho_j^{T_k})\rangle.
  \end{align*}
  Now notice that the right hand sides of these are equal for all $j\notin\supp T$ by the previously obtained identity.  Since the Ringel-Euler form is non-degenerate this implies $\pi_T^c(\lambda_i^{T_k})-\pi_T^c(\lambda_i^{T_k^*})=-[-b_{ik}^{T_k}]_+d_i/d_k\pi_T^c(\lambda_k^{T_k})$.
  Similarly we obtain identities relating $\pi_T(\rho_i^{T_k})$ and $\pi_T(\rho_i^{T_k^*})$:
  \begin{align*}
   &\langle[T_j],\pi_T(\rho_i^{T_k})-\pi_T(\rho_i^{T_k^*})\rangle=0 \text{ and }\\ 
   &\langle[T_k],\pi_T(\rho_i^{T_k})-\pi_T(\rho_i^{T_k^*})\rangle=-\langle[E]-[T_k^*],\pi_T(\rho_i^{T_k^*})\rangle=-[-b_{ik}^{T_k}]_+d_i,
  \end{align*}
  so that $\pi_T(\rho_i^{T_k})-\pi_T(\rho_i^{T_k^*})=-[-b_{ik}^{T_k}]_+d_i/d_k\pi_T(\rho_k^{T_k})$ and $\pi_T^c(\rho_i^{T_k})-\pi_T^c(\rho_i^{T_k^*})=-[-b_{ik}^{T_k}]_+d_i/d_k\pi_T^c(\rho_k^{T_k})$ as above.  Notice that for $i\notin\supp T$ we have
  \begin{align*}
   &\langle\pi_T(\lambda_i^{T_k})-\pi_T(\lambda_i^{T_k^*}),[T_j]\rangle=0 \text{ and }\\
   &\langle\pi_T(\lambda_i^{T_k})-\pi_T(\lambda_i^{T_k^*}),[T_k]\rangle=\langle S_i,[T_k]\rangle-\langle S_i,[E]-[T_k^*]\rangle=0,
  \end{align*}
  so that the non-degeneracy of the Ringel-Euler form implies $\pi_T(\lambda_i^{T_k})-\pi_T(\lambda_i^{T_k^*})=0$.  Similarly one shows that $\pi_T(\rho_i^{T_k})-\pi_T(\rho_i^{T_k^*})=0$.  Note that we may write these identities in the $\bar{T}\oplus T_k$-basis and equate coefficients of $[T_k]$ to get $\ell_{ik}^{T_k}=-\ell_{ik}^{T_k^*}$ and $r_{ik}^{T_k}=-r_{ik}^{T_k^*}$, or equivalently $r_{ki}^{T_k}=-r_{ki}^{T_k^*}$ and $\ell_{ki}^{T_k}=-\ell_{ki}^{T_k^*}$.  In other words, we have $\pi_T^c(\rho_k^{T_k})=-\pi_T^c(\rho_k^{T_k^*})$ and $\pi_T^c(\lambda_k^{T_k})=-\pi_T^c(\lambda_k^{T_k^*})$.  For $i\notin\supp T$ we have the following:
  \begin{equation*}
   \langle\pi_T(\lambda_i^{T_k^*}),\pi_T(\rho_k^{T_k^*})\rangle=\langle\pi_T(\lambda_i^{T_k^*}),[T_k^*]-[D]\rangle=\langle S_i,[T_k^*]-[D]\rangle=\langle S_i,[C]\rangle.
  \end{equation*}
  On the other hand this is equal to $\ell_{ik}^{T_k^*}d_k=r_{ki}^{T_k^*}d_i$ and thus $\pi_T^c(\rho_k^{T_k^*})={}^{*_P}[C]$.  A similar computation with $\langle\pi_T(\lambda_k^{T_k}),\pi_T(\rho_i^{T_k})\rangle$ gives $\langle [B],S_i\rangle=r_{ik}^{T_k}d_k=\ell_{ki}^{T_k}d_i$ and $\pi_T^c(\lambda_k^{T_k})=[B]^{*_P}$.  Finally we note that $\pi_T^c(\lambda_i^{T_k})$ and $\pi_T^c(\lambda_i^{T_k^*})$ have the same defining equations.  Therefore we have $\pi_T^c(\lambda_i^{T_k})=\pi_T^c(\lambda_i^{T_k^*})$ and similarly $\pi_T^c(\rho_i^{T_k})=\pi_T^c(\rho_i^{T_k^*})$.
 \end{proof}
 \noindent The following Proposition is based upon \cite[Proposition 23]{hub1}.
 \begin{proposition}
  Thinking of the $k^{th}$ column $\bfb^k$ of $B_T$ as representing an element of $\cK(Q)$ written in the $\bar{T}\oplus T_k$-basis, we may write
  \begin{equation*}
   \bfb^k=\rho_k^{T_k}-\lambda_k^{T_k}.
  \end{equation*}
  Similarly, thinking of the $k^{th}$ column $\bfb^k$ of $B_{T_k^*\oplus\bar{T}}$ as representing an element of $\cK(Q)$ written in the $T_k^*\oplus\bar{T}$-basis, we may write
  \begin{equation*}
   \bfb^k=\rho_k^{T_k^*}-\lambda_k^{T_k^*}.
  \end{equation*}
 \end{proposition}
 \begin{proof}
  According to Lemma~\ref{lem:24} we may write 
  \begin{align*}
   \rho_k^{T_k}-\lambda_k^{T_k}
   &=[D]-[T_k^*]-{}^{*_P}[C]-[T_k]+[A]-[B]^{*_P}\\
   &=[A]+[D]-[E]+{}^{*_P}[\tau B]-{}^{*_P}[C]\\
   &=[A]+[D]+{}^{*_P}[I]-[E],
  \end{align*}
  but this is exactly $\bfb^k$.  From above we see that $\rho_k^{T_k^*}=-\rho_k^{T_k}$ and $\lambda_k^{T_k^*}=-\lambda_k^{T_k}$, the result for $B_{T_k^*\oplus\bar{T}}$ follows.
 \end{proof}

 Let $\bar{T}$ be a local tilting representation with $k\notin\supp(\bar{T})$ and write $T=\bar{T}\oplus T_k^*$ where $T_k^*$ is the unique compliment to $\bar{T}$.  We will write $\lambda_i^{T_k^*}$, $\rho_i^{T_k^*}$, $\lambda_i^k$ and $\rho_i^k$ for the duals with respect to the $\bar{T}\oplus T_k^*$-basis and the $\bar{T}$-basis respectively.  We again recall the matrix $B_T=(b_{ij}^{T_k^*})$ defined in Section~\ref{sec:tilt_matrix}.  The following Lemma is based upon \cite[Lemma 28]{hub1}.
 \begin{lemma}\label{lem:28}
  We may compute $\pi_T(\lambda_k^{T_k^*})$, $\pi_{\bar{T}}(\lambda_k^k)$, $\pi_T(\rho_k^{T_k^*})$, and $\pi_{\bar{T}}(\rho_k^k)$ as follows:
  \begin{align*}
   \pi_T(\lambda_k^{T_k^*})=[T_k^*]-[F] &\text{ and } \pi_T(\rho_k^{T_k^*})=[T_k^*]-[G],\\
   \pi_{\bar{T}}(\lambda_k^k)=-[\rad\bar{P}_k] &\text{ and } \pi_{\bar{T}}(\rho_k^k)=-[\bar{I}_k/S_k].
  \end{align*}
  We may also compute $\pi_T^c(\lambda_k^{T_k^*})$, $\pi_{\bar{T}}^c(\lambda_k^k)$, $\pi_T^c(\rho_k^{T_k^*})$, and $\pi_{\bar{T}}^c(\rho_k^k)$ as follows:
  \begin{align*}
   \pi_T^c(\lambda_k^{T_k^*})=-[P']^{*_P} &\text{ and } \pi_T^c(\rho_k^{T_k^*})=-{}^{*_P}[I'],\\
   \pi_{\bar{T}}^c(\lambda_k^k)=[P_k]-{}^{*_P}[I'] &\text{ and } \pi_{\bar{T}}^c(\rho_k^k)=[P_k]-[P']^{*_P}.
  \end{align*}
  Moreover, for $i\in\supp\bar{T}$ we have
  \begin{align*}
   \pi_{\bar{T}}(\lambda_i^k)&=\pi_T(\lambda_i^{T_k^*})+[-b_{ik}^{T_k^*}]_+d_i/d_k\pi_T(\lambda_k^{T_k^*}) &\text{ and } &&\pi_{\bar{T}}(\rho_i^k)&=\pi_T(\rho_i^{T_k^*})+[b_{ik}^{T_k^*}]_+d_i/d_k\pi_T(\rho_k^{T_k^*}),\\
   \pi_{\bar{T}}^c(\lambda_i^k)&=\pi_T^c(\lambda_i^{T_k^*})+[-b_{ik}^{T_k^*}]_+d_i/d_k\pi_T^c(\lambda_k^{T_k^*}) &\text{ and } &&\pi_{\bar{T}}^c(\rho_i^k)&=\pi_T^c(\rho_i^{T_k^*})+[b_{ik}^{T_k^*}]_+d_i/d_k\pi_T^c(\rho_k^{T_k^*}),
  \end{align*}
  and for $i\notin\supp T$ we have
  \begin{align*}
   \pi_{\bar{T}}(\lambda_i^k)&=\pi_T(\lambda_i^{T_k^*})+[-b_{ik}^{T_k^*}]_+d_i/d_k\pi_T(\lambda_k^{T_k^*}) &\text{ and } &&\pi_{\bar{T}}(\rho_i^k)&=\pi_T(\rho_i^{T_k^*})+[b_{ik}^{T_k^*}]_+d_i/d_k\pi_T(\rho_k^{T_k^*}),\\
   \pi_{\bar{T}}^c(\lambda_i^k)&=\pi_T^c(\lambda_i^{T_k^*})-[b_{ik}^{T_k^*}]_+d_i/d_k\pi_{\bar{T}}^c(\lambda_k^k) &\text{ and } &&\pi_{\bar{T}}^c(\rho_i^k)&=\pi_T^c(\rho_i^{T_k^*})-[-b_{ik}^{T_k^*}]_+d_i/d_k\pi_{\bar{T}}^c(\rho_k^k).
  \end{align*}
 \end{lemma}
 \begin{proof}
  Applying the functor $\Hom(-,T_j)$ for $j\in\supp(\bar{T})$ to the sequence 
  \begin{center}
   \begin{tikzpicture}
    \matrix (m) [matrix of math nodes, row sep=3em, column sep=2.5em, text height=1.5ex, text depth=0.25ex]
     {0 & \bar{P}_k & T_k^* & F & 0\\};
    \path[->]
     (m-1-1) edge (m-1-2)
     (m-1-2) edge (m-1-3)
     (m-1-3) edge (m-1-4)
     (m-1-4) edge (m-1-5);
   \end{tikzpicture}
  \end{center}
  shows that $\Hom(F,T_j)\cong\Hom(T_k^*,T_j)$ and thus $\langle F,T_j\rangle=\langle T_k^*,T_j\rangle$.  Similarly we may apply $\Hom(T_j,-)$ to the sequence
  \begin{center}
   \begin{tikzpicture}
    \matrix (m) [matrix of math nodes, row sep=3em, column sep=2.5em, text height=1.5ex, text depth=0.25ex]
     {0 & G & T_k^* & \bar{I}_k & 0\\};
    \path[->]
     (m-1-1) edge (m-1-2)
     (m-1-2) edge (m-1-3)
     (m-1-3) edge (m-1-4)
     (m-1-4) edge (m-1-5);
   \end{tikzpicture}
  \end{center}
  to see that $\langle T_j,G\rangle=\langle T_j,T_k^*\rangle$.  From the proof of Proposition~\ref{prop:exch2} we have that $\Hom(F,T_k^*)=0$ and by a similar argument $\Hom(T_k^*,G)=0$.  Since $F,G\in add(\bar{T})$ we may identify $\pi_T(\lambda_k^{T_k^*})=[T_k^*]-[F]$ and $\pi_T(\rho_k^{T_k^*})=[T_k^*]-[G]$.  

  Recall that we have the following short exact sequences defining $P'$ and $F$:
  \begin{center}
   \begin{tikzpicture}
    \matrix (m) [matrix of math nodes, row sep=1em, column sep=2.5em, text height=1.5ex, text depth=0.25ex]
     {0 & \bar{P}_k & T_k^* & F & 0\,,\\ 0 & P' & P_k & \bar{P}_k & 0\,.\\};
    \path[->]
     (m-1-1) edge (m-1-2)
     (m-1-2) edge (m-1-3)
     (m-1-3) edge (m-1-4)
     (m-1-4) edge (m-1-5)
     (m-2-1) edge (m-2-2)
     (m-2-2) edge (m-2-3)
     (m-2-3) edge (m-2-4)
     (m-2-4) edge (m-2-5);
   \end{tikzpicture}
  \end{center}
  Applying $\Hom(P_i,-)$ to the first sequence shows that $\Hom(P_i,\bar{P}_k)=0$ for $i\notin\supp T$ and thus applying the same functor to the second sequence gives $\Hom(P_i,P')\cong\Hom(P_i,P_k)$ for $i\notin\supp T$.  By noting that $\Hom(P_k,P')=0$ we see that $\pi_{\bar{T}}^c(\rho_k^k)=[P_k]-[P']=[P_k]-[P']^{*_P}$.  Since $k\notin\supp(\bar{T})$ we have $\Hom(\bar{P}_k,\bar{T})=0$, in particular $\langle\bar{P}_k,\bar{T}\rangle=0$.  Thus we see that $-\langle\rad\bar{P}_k,T_j\rangle=\langle S_k,T_j\rangle$ for $j\in\supp(\bar{T})$ and hence $\pi_{\bar{T}}(\lambda_k^k)=-[\rad\bar{P}_k]\in\cK(Q_{\bar{T}})$.

  Similarly we have the following short exact sequences defining $G$ and $I'$:
  \begin{center}
   \begin{tikzpicture}
    \matrix (m) [matrix of math nodes, row sep=1em, column sep=2.5em, text height=1.5ex, text depth=0.25ex]
     {0 & G & T_k^* & \bar{I}_k & 0\,,\\ 0 & \bar{I}_k & I_k & I' & 0\,.\\};
    \path[->]
     (m-1-1) edge (m-1-2)
     (m-1-2) edge (m-1-3)
     (m-1-3) edge (m-1-4)
     (m-1-4) edge (m-1-5)
     (m-2-1) edge (m-2-2)
     (m-2-2) edge (m-2-3)
     (m-2-3) edge (m-2-4)
     (m-2-4) edge (m-2-5);
   \end{tikzpicture}
  \end{center}
  Applying the functor $\Hom(-,I_j)$ we see that $\Hom(I_k,I_j)\cong\Hom(I',I_j)$ for all $j\notin\supp T$.  Recall that the inverse Nakayama functor $\nu^{-1}$ is an equivalence of categories from the full subcategory of $\Rep_\FF(Q,\bfd)$ consisting of injective objects to the full subcategory consisting of projective objects.  Thus we see that 
  \begin{align*}
   \Hom(P_k,P_j)
   &=\Hom(\nu^{-1}I_k,\nu^{-1}I_j)\\
   &\cong\Hom(\nu^{-1}I',\nu^{-1}I_j)=\Hom(\nu^{-1}I',P_j)
  \end{align*}
  and $\Hom(\nu^{-1}I',P_k)=\Hom(\nu^{-1}I',\nu^{-1}I_k)=0$.  Thus we see that $\pi_{\bar{T}}^c(\lambda_k^k)=[P_k]-[\nu^{-1}I']=[P_k]-{}^{*_P}[I']$.  Since $k\notin\supp(\bar{T})$ we have $\Hom(\bar{T},\bar{I}_k)=0$, in particular $\langle\bar{T},\bar{I}_k\rangle=0$.  Thus we see that $-\langle T_j,\bar{I}_k/S_k\rangle=\langle T_j,S_k\rangle$ for $j\in\supp(\bar{T})$ and hence $\pi_{\bar{T}}(\rho_k^k)=-[\bar{I}_k/S_k]\in\cK(Q_{\bar{T}})$.  To obtain the remaining duals $\pi_T^c(\rho_k^{T_k^*})$ and $\pi_T^c(\lambda_k^{T_k^*})$ we assume $i\notin\supp T$ and compute
  \begin{equation*}
   \langle\pi_T(\lambda_i^{T_k^*}),\pi_T(\rho_k^{T_k^*})\rangle=\langle\pi_T(\lambda_i^{T_k^*}),[T_k^*]-[G]\rangle=\langle S_i,[T_k^*]-[G]\rangle=\langle S_i,[I_k]-[I']\rangle=-\langle S_i,[I']\rangle.
  \end{equation*}
  On the other hand this is equal to $\ell_{ik}^{T_k^*}d_k=r_{ki}^{T_k^*}d_i$ and thus $\pi_T^c(\rho_k^{T_k^*})=-{}^{*_P}[I']$.  A similar computation with $\langle\pi_T(\lambda_k^{T_k^*}),\pi_T(\rho_i^{T_k^*})\rangle$ gives $-\langle [P'],S_i\rangle=r_{ik}^{T_k^*}d_k=\ell_{ki}^{T_k^*}d_i$ and $\pi_T^c(\lambda_k^{T_k^*})=-[P']^{*_P}$.
    
  For $i\in\supp\bar{T}$ we may compare $\pi_T(\lambda_i^{T_k^*})$ and $\pi_{\bar{T}}(\lambda_i^k)$ as follows:
  \begin{align*}
   &\langle\pi_T(\lambda_i^{T_k^*})-\pi_{\bar{T}}(\lambda_i^k),[T_j]\rangle=0 \text{ and }\\
   &\langle\pi_T(\lambda_i^{T_k^*})-\pi_{\bar{T}}(\lambda_i^k),[T_k^*]\rangle=-\langle\pi_{\bar{T}}(\lambda_i^k),[T_k^*]\rangle=-\langle\pi_{\bar{T}}(\lambda_i^k),[G]\rangle=-[-b_{ik}^{T_k^*}]_+d_i,
  \end{align*}
  so that $\pi_T(\lambda_i^{T_k^*})-\pi_{\bar{T}}(\lambda_i^k)=-[-b_{ik}^{T_k^*}]_+d_i/d_k\pi_T(\lambda_k^{T_k^*})$.  Similarly we may compare $\pi_T(\rho_i^{T_k^*})$ and $\pi_{\bar{T}}(\rho_i^k)$ as follows:
  \begin{align*}
   &\langle[T_j],\pi_T(\rho_i^{T_k^*})-\pi_{\bar{T}}(\rho_i^k)\rangle=0 \text{ and }\\
   &\langle[T_k^*],\pi_T(\rho_i^{T_k^*})-\pi_{\bar{T}}(\rho_i^k)\rangle=-\langle[T_k^*],\pi_{\bar{T}}(\rho_i^k)\rangle=-\langle[F],\pi_{\bar{T}}(\rho_i^k)\rangle=-[b_{ik}^{T_k^*}]_+d_i,
  \end{align*}
  so that $\pi_T(\rho_i^{T_k^*})-\pi_{\bar{T}}(\rho_i^k)=-[b_{ik}^{T_k^*}]_+d_i/d_k\pi_T(\rho_k^{T_k^*})$.  Now we suppose $i\notin\supp T$.  Then we may compare $\pi_T(\lambda_i^{T_k^*})$ and $\pi_{\bar{T}}(\lambda_i^k)$ as follows:
  \begin{align*}
   &\langle\pi_T(\lambda_i^{T_k^*})-\pi_{\bar{T}}(\lambda_i^k),[T_j]\rangle=0 \text{ and }\\
   &\langle\pi_T(\lambda_i^{T_k^*})-\pi_{\bar{T}}(\lambda_i^k),[T_k^*]\rangle=\langle S_i,[T_k^*]\rangle-\langle\pi_{\bar{T}}(\lambda_i^k),[G]\rangle=\langle S_i,[T_k^*]-[G]\rangle=\langle S_i,I_k-I'\rangle=-[-b_{ik}^{T_k^*}]_+d_i,
  \end{align*}
  so that $\pi_T(\lambda_i^{T_k^*})-\pi_{\bar{T}}(\lambda_i^k)=-[-b_{ik}^{T_k^*}]_+d_i/d_k\pi_T(\lambda_k^{T_k^*})$.  For $j\in\supp\bar{T}$ we consider the coefficient of $[T_j]$ in this identity to get
  \begin{align*}
   \ell_{ij}^{T_k^*}-\ell_{ij}^k&=-[-b_{ik}^{T_k^*}]_+d_i/d_k\ell_{kj}^{T_k^*}\\
   \ell_{ij}^{T_k^*}d_j-\ell_{ij}^kd_j&=-[-b_{ik}^{T_k^*}]_+d_i/d_k\ell_{kj}^{T_k^*}d_j\\
   r_{ji}^{T_k^*}d_i-r_{ji}^kd_i&=r_{ki}^{T_k^*}d_i/d_k\ell_{kj}^{T_k^*}d_j\\
   r_{ji}^{T_k^*}-r_{ji}^k&=\ell_{kj}^{T_k^*}d_j/d_kr_{ki}^{T_k^*},
  \end{align*}
  which we may recognize as giving the coefficient of $[P_i]$ in the identity $\pi_T^c(\rho_j^{T_k^*})-\pi_{\bar{T}}^c(\rho_j^k)=-[b_{jk}]_+d_j/d_k\pi_T^c(\rho_k^{T_k^*})$.  Similarly we may compare $\pi_T(\rho_i^{T_k^*})$ and $\pi_{\bar{T}}(\rho_i^k)$ as follows:
  \begin{align*}
   &\langle[T_j],\pi_T(\rho_i^{T_k^*})-\pi_{\bar{T}}(\rho_i^k)\rangle=0 \text{ and }\\
   &\langle[T_k^*],\pi_T(\rho_i^{T_k^*})-\pi_{\bar{T}}(\rho_i^k)\rangle=\langle[T_k^*],S_i\rangle-\langle[F],\pi_{\bar{T}}(\rho_i^k)\rangle=\langle[T_k^*]-[F],S_i\rangle=\langle P_k-P',S_i\rangle=-[b_{ik}^{T_k^*}]_+d_i,
  \end{align*}
  so that $\pi_T(\rho_i^{T_k^*})-\pi_{\bar{T}}(\rho_i^k)=-[b_{ik}^{T_k^*}]_+d_i/d_k\pi_T(\rho_k^{T_k^*})$.  For $j\in\supp\bar{T}$ we again consider the coefficient of $[T_j]$ in this identity to get
  \begin{align*}
   r_{ij}^{T_k^*}-r_{ij}^k&=-[b_{ik}^{T_k^*}]_+d_i/d_kr_{kj}^{T_k^*}\\
   r_{ij}^{T_k^*}d_j-r_{ij}^kd_j&=-[b_{ik}^{T_k^*}]_+d_i/d_kr_{kj}^{T_k^*}d_j\\
   \ell_{ji}^{T_k^*}d_i-\ell_{ji}^kd_i&=\ell_{ki}^{T_k^*}d_i/d_kr_{kj}^{T_k^*}d_j\\
   \ell_{ji}^{T_k^*}-\ell_{ji}^k&=r_{kj}^{T_k^*}d_j/d_k\ell_{ki}^{T_k^*},
  \end{align*}
  which we may recognize as giving the coefficient of $[P_i]$ in the identity $\pi_T^c(\lambda_j^{T_k^*})-\pi_{\bar{T}}^c(\lambda_j^k)=-[-b_{jk}]_+d_j/d_k\pi_T^c(\lambda_k^{T_k^*})$.  We may also compare $\pi_T^c(\lambda_i^{T_k^*})$ and $\pi_{\bar{T}}^c(\lambda_i^k)$ as follows:
  \begin{align*}
   &\langle\pi_T^c(\lambda_i^{T_k^*})-\pi_{\bar{T}}^c(\lambda_i^k),[P_j]\rangle=0 \text{ and }\\
   &\langle\pi_T^c(\lambda_i^{T_k^*})-\pi_{\bar{T}}^c(\lambda_i^k),[P_k]\rangle=\langle\pi_T^c(\lambda_i^{T_k^*}),[P_k]\rangle=\langle\pi_T^c(\lambda_i^{T_k^*}),[P']\rangle=[b_{ik}^{T_k^*}]_+d_i,
  \end{align*}
  so that $\pi_T^c(\lambda_i^{T_k^*})-\pi_{\bar{T}}^c(\lambda_i^k)=[b_{ik}^{T_k^*}]_+d_i/d_k\pi_{\bar{T}}^c(\lambda_k^k)$.  Finally we may compare $\pi_T^c(\rho_i^{T_k^*})$ and $\pi_{\bar{T}}^c(\rho_i^k)$ as follows:
  \begin{align*}
   &\langle[P_j],\pi_T^c(\rho_i^{T_k^*})-\pi_{\bar{T}}^c(\rho_i^k)\rangle=0 \text{ and }\\
   &\langle[P_k],\pi_T^c(\rho_i^{T_k^*})-\pi_{\bar{T}}^c(\rho_i^k)\rangle=\langle[P_k],\pi_T^c(\rho_i^{T_k^*})\rangle=\langle[\nu^{-1}I'],\pi_T^c(\rho_i^{T_k^*})\rangle=[-b_{ik}^{T_k^*}]_+d_i,
  \end{align*}
  so that $\pi_T^c(\rho_i^{T_k^*})-\pi_{\bar{T}}^c(\rho_i^k)=[-b_{ik}^{T_k^*}]_+d_i/d_k\pi_{\bar{T}}^c(\rho_k^k)$.  
 \end{proof}
 \noindent The following Proposition is based upon \cite[Proposition 27]{hub1}.
 \begin{proposition}\label{prop:27}
  Thinking of the $k^{th}$ column $\bfb^k$ of $B_T$ as representing an element of $\cK(Q)$ written in the $\bar{T}\oplus T_k^*$-basis, we may write
  \begin{equation*}
   \bfb^k=\rho_k^{T_k^*}-\lambda_k^{T_k^*}.
  \end{equation*}
  Similarly, thinking of the $k^{th}$ column $\bfb^k$ of $B_{\bar{T}}$ as representing an element of $\cK(Q)$ written in the $\bar{T}$-basis, we may write
  \begin{equation*}
   \bfb^k=\rho_k^k-\lambda_k^k.
  \end{equation*}
 \end{proposition}
 \begin{proof}
  According to Lemma~\ref{lem:28} we may write 
  \begin{align*}
   \rho_k^{T_k}-\lambda_k^{T_k}
   &=[T_k^*]-[G]-{}^{*_P}[I']-[T_k^*]+[F]+[P']^{*_P}\\
   &=[F]+[P']^{*_P}-[G]-{}^{*_P}[I'],
  \end{align*}
  but this is exactly the $k^{th}$ column of $B_T$.  Similarly we have
  \begin{align*}
   \rho_k^k-\lambda_k^k
   &=-[\bar{I}_k/S_k]+[P_k]-[P']^{*_P}+[\rad\bar{P}_k]-[P_k]+{}^{*_P}[I']\\
   &=-[\bar{I}_k]+[S_k]-[P']^{*_P}+[\bar{P}_k]-[S_k]+{}^{*_P}[I']\\
   &=-[T_k^*]+[G]-[P']^{*_P}+[T_k^*]-[F]+{}^{*_P}[I']\\
   &=[G]+{}^{*_P}[I']-[F]-[P']^{*_P},
  \end{align*}
  but this is exactly the $k^{th}$ column of $B_{\bar{T}}$.
 \end{proof}

 We are finally ready to show that mutation of tilting pairs corresponds to mutation of exchange matrices.
 \begin{theorem}\label{th:tilt_matrix_mut}
  Suppose $\mu_k(T)=T'$.  Then $B_T$ and $B_{T'}$ are related by Fomin-Zelevinsky matrix mutation in direction $k$.
 \end{theorem}
 \begin{proof}
  Note that we have shown for any vertex $i\in Q_0$ we have $\bfb^i=\rho_i^{T_k}-\lambda_i^{T_k}$ for $k\in\supp \bar{T}$ and $\bfb^i=\rho_i^k-\lambda_i^k$ when $k\notin\supp(T)$.  In particular, this implies that the matrices $B_T$ and $B_{T'}$ are skew-symmetrizable, i.e. $DB_T$ and $DB_{T'}$ are skew-symmetric.  We will have two cases to consider.  

  \emph{Case 1:} Suppose we have $k\in\supp(T)=\supp(\bar{T})$.  Write $T_k^*\not\cong T_k$ for the unique compliment of $\bar{T}$.  We will show that $B_{\bar{T}\oplus T_k^*}=EB_TF$ where the matrices $E$ and $F$ are given by \eqref{eq:EF}.  We will consider the $j^{th}$ column $\bfb^j=B_T^j$ of $B_T$ as an element of $\cK(Q)$ via $\bfb^j=\sum\limits_{i\in\supp T} b_{ij}^{T_k}[T_i]+\sum\limits_{i\notin\supp T} b_{ij}^{T_k}[P_i]$.  First note that when $j\notin\supp T$ we have $b_{jk}^{T_k}\ge0$ and so
  \begin{align*}
   (B_TF)^j
   &=\bfb^j+[b_{kj}^{T_k}]_+\bfb^k\\
   &=\bfb^j+[-b_{jk}^{T_k}]_+d_j/d_k\bfb^k\\
   &=\rho_j^{T_k}-\lambda_j^{T_k}\\
   &=\rho_j^{T_k^*}-\lambda_j^{T_k^*}\\
   &=B_{\bar{T}\oplus T_k^*}^j.
  \end{align*}
  Then since $\bfb_-^k:=\sum\limits_{i\in\supp T}[-b_{ik}^{T_k}]_+[T_i]+\sum\limits_{i\notin\supp T}[-b_{ik}^{T_k}]_+[P_i]$ is $[E]$ the short exact sequence 
  \begin{center}
   \begin{tikzpicture}
    \matrix (m) [matrix of math nodes, row sep=3em, column sep=2.5em, text height=1.5ex, text depth=0.25ex]
     {0 & T_k^* & E & T_k & 0\\};
    \path[->]
     (m-1-1) edge (m-1-2)
     (m-1-2) edge (m-1-3)
     (m-1-3) edge (m-1-4)
     (m-1-4) edge (m-1-5);
   \end{tikzpicture}
  \end{center}
  implies that the change from the $\bar{T}\oplus T_k$-basis to the $\bar{T}\oplus T_k^*$-basis is given by left multiplication by the matrix $E$ from equation~\eqref{eq:EF}.  Hence the $j^{th}$ columns of $B_{\bar{T}\oplus T_k}$ and $B_{\bar{T}\oplus T_k^*}$ are related by the Fomin-Zelevinsky matrix mutation.  Now we suppose $j\in\supp T$ and make the following similar computation:
  \begin{align*}
   (B_TF)^j
   &=\bfb^j+[-b_{jk}^{T_k}]_+d_j/d_k\bfb^k\\
   &=\rho_j^{T_k}-\lambda_j^{T_k}+[-b_{jk}^{T_k}]_+d_j/d_k(\rho_k^{T_k}-\lambda_k^{T_k})\\
   &=\rho_j^{T_k^*}-\lambda_j^{T_k^*}\\
   &=B_{\bar{T}\oplus T_k^*}^j.
  \end{align*}
  Again note that the change from the $\bar{T}\oplus T_k$-basis to the $\bar{T}\oplus T_k^*$-basis is given by left multiplication by the matrix $E$ from equation~\eqref{eq:EF}.  This completes the proof in the case $k\in\supp(T)=\supp(\bar{T})$.

  \emph{Case 2:} Suppose $k\notin\supp(T)$.  Then $T$ has a unique complement $T_k^*$ such that $T_k^*\oplus T$ is a local tilting representation and $\supp(T_k^*\oplus T)=\supp(T)\cup\{k\}$.  Note that by the definition of $B_T$ and $B_{\bar{T}}$ their $k^{th}$ columns are related by the Fomin-Zelevinsky matrix mutation.  Thus our goal is to show for $i,j\ne k$ that $b_{ij}^k=b_{ij}^{T_k^*}+\delta_{ij}^{T_k^*}$ for $\delta_{ij}^{T_k^*}=[b_{ik}^{T_k^*}]_+b_{kj}^{T_k^*}+b_{ik}^{T_k^*}[-b_{kj}^{T_k^*}]_+$.  We have three sub-cases to consider depending on the positions of $i$ and $j$.
  \begin{enumerate}
   \item Suppose $j\in\supp\bar{T}$ and $i\ne k$.  Then according to Lemma~\ref{lem:28} we have
    \begin{align*}
     b_{ij}^k
     &= r_{ji}^k-\ell_{ji}^k\\
     &= r_{ji}^{T_k^*}+[b_{jk}^{T_k^*}]_+d_j/d_kr_{ki}^{T_k^*}-\ell_{ji}^{T_k^*}-[-b_{jk}^{T_k^*}]_+d_j/d_k\ell_{ki}^{T_k^*}\\
     &= b_{ij}^{T_k^*}+[-b_{kj}^{T_k^*}]_+r_{ki}^{T_k^*}-[b_{kj}^{T_k^*}]_+\ell_{ki}^{T_k^*}\\
     &= b_{ij}^{T_k^*}-[-b_{kj}^{T_k^*}]_+[-b_{ik}^{T_k^*}]_++[b_{kj}^{T_k^*}]_+[b_{ik}^{T_k^*}]_+\\
     &= b_{ij}^{T_k^*}+[b_{ik}^{T_k^*}]_+[b_{kj}^{T_k^*}]_+-[b_{ik}^{T_k^*}]_+[-b_{kj}^{T_k^*}]_++[b_{ik}^{T_k^*}]_+[-b_{kj}^{T_k^*}]_+-[-b_{ik}^{T_k^*}]_+[-b_{kj}^{T_k^*}]_+\\
     &= b_{ij}^{T_k^*}+[b_{ik}^{T_k^*}]_+([b_{kj}^{T_k^*}]_+-[-b_{kj}^{T_k^*}]_+)+([b_{ik}^{T_k^*}]_+-[-b_{ik}^{T_k^*}]_+)[-b_{kj}^{T_k^*}]_+\\
     &= b_{ij}^{T_k^*}+[b_{ik}^{T_k^*}]_+b_{kj}^{T_k^*}+b_{ik}^{T_k^*}[-b_{kj}^{T_k^*}]_+.
    \end{align*}
  
   \item Suppose $j\notin\supp T$ and $i\in\supp\bar{T}$.  Following Lemma~\ref{lem:28} we have
    \begin{align*}
     b_{ij}^k=r_{ji}^k-\ell_{ji}^k=r_{ji}^{T_k^*}+[b_{jk}^{T_k^*}]_+d_j/d_kr_{ki}^{T_k^*}-\ell_{ji}^{T_k^*}-[-b_{jk}^{T_k^*}]_+d_j/d_k\ell_{ki}^{T_k^*}
    \end{align*}
    and the proof proceeds as in case (1).
  
   \item Finally suppose $i,j\notin\supp T$.  Then Lemma~\ref{lem:28} gives
    \begin{align*}
     b_{ij}^k
     &= r_{ji}^k-\ell_{ji}^k\\
     &= r_{ji}^{T_k^*}-[-b_{jk}^{T_k^*}]_+d_j/d_kr_{ki}^k-\ell_{ji}^{T_k^*}+[b_{jk}^{T_k^*}]_+d_j/d_k\ell_{ki}^k\\
     &= b_{ij}^{T_k^*}-r_{ki}^k[b_{kj}^{T_k^*}]_++\ell_{ki}^k[-b_{kj}^{T_k^*}]_+\\
     &= b_{ij}^{T_k^*}+[b_{ik}^{T_k^*}]_+[b_{kj}^{T_k^*}]_+-[-b_{ik}^{T_k^*}]_+[-b_{kj}^{T_k^*}]_+\\
     &= b_{ij}^{T_k^*}+[b_{ik}^{T_k^*}]_+[b_{kj}^{T_k^*}]_+-[b_{ik}^{T_k^*}]_+[-b_{kj}^{T_k^*}]_++[b_{ik}^{T_k^*}]_+[-b_{kj}^{T_k^*}]_+-[-b_{ik}^{T_k^*}]_+[-b_{kj}^{T_k^*}]_+\\
     &= b_{ij}^{T_k^*}+[b_{ik}^{T_k^*}]_+([b_{kj}^{T_k^*}]_+-[-b_{kj}^{T_k^*}]_+)+([b_{ik}^{T_k^*}]_+-[-b_{ik}^{T_k^*}]_+)[-b_{kj}^{T_k^*}]_+\\
     &= b_{ij}^{T_k^*}+[b_{ik}^{T_k^*}]_+b_{kj}^{T_k^*}+b_{ik}^{T_k^*}[-b_{kj}^{T_k^*}]_+.
    \end{align*}
  \end{enumerate}
  This completes the proof.
 \end{proof}

 \section{Quantum Seeds Associated to Local Tilting Representations}\label{sec:quantum_seed}  In this section we assign a quantum seed $\Sigma_T=(\bfX_T,B_T)$ to each local tilting representation $T$.  Combining with Theorem~\ref{th:tilt_matrix_mut} we will complete the proof that all cluster variables are given by the quantum cluster character applied to an exceptional representation of $(Q,\bfd)$.

 Recall that starting with a principal compatible pair $(\tilde{B},\Lambda)$ we construct the quantum cluster algebra $\cA_{|\FF|}(\tilde{B},\Lambda)$ with initial cluster $\{X_1,\ldots,X_m\}$ and a valued quiver $(Q,\bfd)$ with principal frozen vertices.  For a local tilting representation $T$ we define the cluster $\bfX_T=\{X'_1,\ldots,X'_m\}$ as follows:
 \[X'_i=\begin{cases}X_i & \text{ if $i\notin\supp T$,}\\ X_{T_i} & \text{ if $i\in\supp T$.}\end{cases}\]
 Recall that the exchange matrix $B_T$ was defined in Section~\ref{sec:tilt_matrix}.  We will only consider those local tilting representations which may be obtained from the trivial local tilting representation $T_0=0$ by a sequence of mutations.  

 \begin{proposition}
  The matrix $B_{T_0}$ is the initial exchange matrix $\tilde{B}$.
 \end{proposition}
 \begin{proof}
  Notice that $\langle P_i,S_j\rangle=\delta_{ij}d_i$ so that $\rho_j=S_j$.  On the other hand we may write
  \begin{equation*}
   r_{ji}d_i=\langle\rho_j,S_i\rangle=\langle S_j,S_i\rangle=-[d_jb_{ji}]_+=-[-d_ib_{ij}]_+,
  \end{equation*}
  so that $r_{ji}=-[-b_{ij}]_+$ and $\ell_{ji}=r_{ij}d_j/d_i=-[b_{ij}]_+$.  Following Proposition~\ref{prop:27} we may compute the $ij$-entry of $B_{T_0}$ as $b_{ij}^{T_0}=r_{ji}-\ell_{ji}=[b_{ij}]_+-[-b_{ij}]_+=b_{ij}$.
 \end{proof}

 Recall that the cluster $\bfX_T$ should consist of quasi-commuting elements.  In Section~\ref{sec:comm}. we have already computed the commutation for the cluster $\bfX_T$ and thus we have defined the commutation matrix $\Lambda_T$ intrinsically in terms of the local tilting representation $T$.  To give $\Lambda_T$ explicitly, we write $I_i$ for the injective hull of the simple valued representation $S_i$, and write $\bfi_i$ for its dimension vector.  Also we will write $\bft_i$ for the dimension vector of the summand $T_i$ of $T$.
 \begin{proposition}
  Let $T$ be a local tilting representation.  Then $\Lambda_T=(\lambda'_{ij})$ is given by
  \[\lambda'_{ij}=\begin{cases}\Lambda({}^*\bfi_i,{}^*\bfi_j) & \text{ if $i,j\notin\supp T$,}\\-\Lambda({}^*\bfi_i,{}^*\bft_j) & \text{ if $i\notin\supp T$ and $j\in\supp T$,}\\-\Lambda({}^*\bft_i,{}^*\bfi_j) & \text{ if $i\in\supp T$ and $j\notin\supp T$,}\\ \Lambda({}^*\bft_i,{}^*\bft_j) & \text{ if $i,j\in\supp T$.}\\\end{cases}\]
 \end{proposition}
 Following Remark~\ref{rem:init_cluster_comm} we have $\Lambda_{T_0}=\Lambda$ so that $\Sigma_{T_0}=(\bfX_{T_0},B_{T_0})$ forms the initial seed for the quantum cluster algebra $\cA_{|\FF|}(\tilde{B},\Lambda)$.  We now obtain the main technical result of this article.
 \begin{theorem}\label{th:tilt_mut}
  Suppose $\mu_k(T)=T'$.  Then the quantum seeds $\Sigma_T$ and $\Sigma_{T'}$ are related by the quantum seed mutation in direction $k$.
 \end{theorem}
 \begin{proof}
  Note that the multiplication formulas of Section~\ref{sec:mult_theorems} exactly correspond to the Berenstein-Zelevinsky quantum cluster exchange relations relating $\bfX_T$ and $\bfX_{T'}$.  We have already shown in Theorem~\ref{th:tilt_matrix_mut} that $\mu_k(B_T)=B_{T'}$.  To complete the proof we remark that the compatibility of the pair $(\tilde{B},\Lambda)$ guarantees that each cluster consists of a quasi-commuting collection of cluster variables and that the commutation matrices of neighboring clusters are related by the Berenstein-Zelevinsky mutation rule.
 \end{proof}
 Since the quantum seed $\Sigma_{T_0}$ identifies with the initial quantum seed of $\cA_{|\FF|}(\tilde{B},\Lambda)$ this completes our goal of generalizing the classical cluster characters of Caldero and Chapoton to the quantum cluster algebra setting.
 \begin{corollary}\cite[Conjecture 1.10]{rupel}\label{cor:main}
  The quantum cluster character $V\mapsto X_V$ defines a bijection from exceptional representations $V$ of $(Q,\bfd)$ to non-initial quantum cluster variables of the quantum cluster algebra $\cA_{|\FF|}(\tilde{B},\Lambda)$.
 \end{corollary} 
 \noindent This completes the proof of Theorem~\ref{th:main} from Section~\ref{sec:intro}.
 
 Note that since there is a unique isomorphism class for each exceptional valued representation, the isomorphism classes of rigid objects in the Grothendieck group $\cK(Q)$ are independent of the choice of ground field $\FF$.  This corollary together with the specialization argument of \cite[Proposition 2.2.3]{qin} immediately implies the following
 \begin{corollary}
  Let $V$ be a rigid representation in $\Rep_\FF(Q,\bfd)$.  Then for any $\bfe\in\cK(Q)$ the Grassmannian $Gr_\bfe^V$ has a counting polynomial $P_\bfe^\bfv(q)$ such that $|Gr_\bfe^V|=P_\bfe^\bfv(|\FF|)$. 
 \end{corollary}
 \begin{proof}
  In \cite[Proposition 2.2.3]{qin} Qin shows that there is a support preserving surjection from $\cA_q(\tilde{B},\Lambda)$, where $q$ is an indeterminate, to $\cA_{|\FF|}(\tilde{B},\Lambda)$.  The quantum Laurent phenomenon \cite{berzel} asserts that the structure constants of the initial cluster expansion of any cluster monomial live in $\ZZ[q^{\pm\half}]$.  Thus for each rigid valued representation $V$ there is a Laurent polynomial $L_\bfe^\bfv(q)$ in $q^\half$ which specializes to $L_\bfe^\bfv(|\FF|)=|\FF|^{-\half\langle\bfe,\bfv-\bfe\rangle}|Gr_\bfe^V|$.  Now since $|Gr_\bfe^V|$ is an integer for all finite fields $\FF$, we see that $P_\bfe^\bfv(q)=q^{\half\langle\bfe,\bfv-\bfe\rangle}L_\bfe^\bfv(q)$ is an honest polynomial in $q$ and $P_\bfe^\bfv(|\FF|)=|Gr_\bfe^V|$.
 \end{proof}
 We conjecture that these polynomials always have positive coefficients. 
 \begin{conjecture}\label{conj:pos}
  For any acyclic valued quiver $(Q,\bfd)$ and any rigid representation $V$, the polynomial $P_\bfe^\bfv(q)\in\ZZ[q]$ has nonnegative integer coefficients.
 \end{conjecture} 
 Qin \cite{qin} has settled this conjecture for acyclic equally valued quivers.  The positivity of these counting polynomials in skew-symmetrizable rank 2 Grassmannians follows from the results of our recent combinatorial description of noncommutative rank 2 cluster variables in \cite{rupel2}.  In a recent preprint \cite{ef}, Efimov shows in the acyclic equally valued case that in addition the counting polynomials are unimodular, i.e. they are a shifted sum of bar-invariant $q$-numbers.  We conjecture that this property always holds.
 \begin{conjecture}
  In the hypotheses of Conjecture~\ref{conj:pos}, the polynomial $P_\bfe^\bfv(q)\in\ZZ[q]$ is unimodular.
 \end{conjecture}

\end{document}